\documentclass[12pt]{amsart}

\usepackage[english]{babel}
\usepackage{amsmath,amsfonts,amsthm,enumerate,amscd,latexsym,curves}
\usepackage{verbatim} 
\usepackage{color}

\usepackage{bm}
\usepackage{graphicx} 
\usepackage[font=footnotesize,labelfont=bf]{caption} 
\usepackage{tikz-cd} 
\usepackage{float} 
\usepackage{subcaption} 

\numberwithin{equation}{section}
\numberwithin{figure}{section}

\usepackage[all]{xy}

\usepackage{setspace}
\usepackage[refpage,intoc]{nomencl}
\usepackage{aliascnt}
\usepackage{hyperref}
\hypersetup{
	colorlinks=true,
	breaklinks=true,
	linkcolor=[rgb]{0.15,0,0.8},
	filecolor=magenta,
	urlcolor=blue,
	citecolor=magenta,
	linktoc=all,
}
\usepackage[nameinlink]{cleveref}

\newtheorem{thm}{Theorem}[section]

\newaliascnt{lemma}{thm}
\newtheorem{lemma}[lemma]{Lemma}
\aliascntresetthe{lemma}
\Crefname{lemma}{Lemma}{Lemmas}

\newaliascnt{prop}{thm}
\newtheorem{prop}[prop]{Proposition}
\aliascntresetthe{prop}
\Crefname{prop}{Proposition}{Propositions}

\newaliascnt{cor}{thm}
\newtheorem{cor}[cor]{Corollary}
\aliascntresetthe{cor}
\Crefname{cor}{Corollary}{Corollaries}

\newaliascnt{conjecture}{thm}
\newtheorem{conjecture}[conjecture]{Conjecture}
\aliascntresetthe{conjecture}
\Crefname{conjecture}{Conjecture}{Conjectures}

\theoremstyle{definition}
\newaliascnt{remark}{thm}
\newtheorem{remark}[remark]{Remark}
\aliascntresetthe{remark}
\Crefname{remark}{Remark}{Remarks}

\newaliascnt{example}{thm}
\newtheorem{example}[example]{Example}
\aliascntresetthe{example}
\Crefname{example}{Example}{Examples}

\newaliascnt{definition}{thm}
\newtheorem{definition}[definition]{Definition}
\aliascntresetthe{definition}
\Crefname{definition}{Definition}{Definitions}

\newaliascnt{notation}{thm}

\aliascntresetthe{notation}
\Crefname{notation}{Notation}{Notations}

\newcommand{\RR}{\mathbb{R}}

\newcommand{\C}{\mathbb{C}}

\newcommand{\ZZ}{\mathbb{Z}}

\newcommand{\Si}{\Sigma}

\newcommand{\ang}{\mathrm{ang}}

\newcommand{\calC}{\mathcal{C}}

\newcommand{\calI}{\mathcal{I}}
\newcommand{\calL}{\mathcal{L}}
\newcommand{\calS}{\mathcal{S}}

\newcommand{\Mod}{\mathrm{Mod}}
\newcommand{\Var}{\mathrm{Var}}
\newcommand{\tat}{t\^ete-\`a-t\^ete }

\usepackage[bindingoffset=0in,
left=1in,
right=1in,
top=1.5in,
bottom=1.5in,
footskip=.85in]{geometry}
\graphicspath{{./figures/}}

\title{Vanishing arcs for isolated plane curve singularities}
\author[Bae, Cho, Choa, Jeong, and Portilla Cuadrado]{Hanwool Bae, Cheol-Hyun 
Cho, Dongwook Choa, Wonbo Jeong, and Pablo Portilla Cuadrado}
\thanks{The first and second authors are supported by the National Research Foundation of Korea(NRF) grant funded by the Korea government(MSIT) (No.2020R1A5A1016126)
and the first author by NRF  0450-20250061. }
\thanks{The third author is supported by Chungbuk National University NUDP program (2025)}
\thanks{The fourth author is supported by G-LAMP program of NRF grant funded by the Ministry of Education (No. RS-2024-00441954)}
\thanks{The last author is supported by RYC2022-035158-I, funded by 
MCIN/AEI/10.13039/501100011033 and by the FSE+}

\makeatletter
\let\c@equation\c@thm
\let\c@figure\c@thm
\makeatother

\begin{document}

\begin{abstract}
The variation operator associated with an isolated hypersurface singularity is 
a classical topological invariant that relates relative and absolute homologies 
of the Milnor fiber via a non trivial isomorphism. Here we work with a 
topological version of this operator that deals with proper arcs and closed 
curves instead of homology cycles. Building on the classical 
framework of 
geometric vanishing cycles, we introduce the concept of vanishing arcsets as 
their counterpart using this geometric variation operator. We characterize 
which properly embedded arcs are sent to geometric vanishing cycles by the 
geometric variation operator in terms of intersections numbers of the arcs and 
their images by the geometric monodromy. Furthermore, we prove that for any 
distinguished 
collection of vanishing cycles arising from an A'Campo’s divide, 
there exists a  topological exceptional collection of arcsets 
whose variation images match this collection.
\end{abstract}
\maketitle

\tableofcontents

\section{Introduction}

This work introduces the notion of {\em vanishing arcs}, a new perspective on 
relative homology classes associated with isolated plane curve singularities. 
Building on the classical framework of vanishing cycles, we define vanishing 
arcs as their counterpart in relative homology via the variation operator.

Let $f:\C^2 \to \C$ be a representative of a germ with an isolated critical 
point at the origin, with Milnor fiber $\Sigma_f$. The variation operator, 
$V_f:H_1(\Sigma_f,\partial \Sigma_f;\ZZ) \to H_1(\Sigma_f;\ZZ)$, makes use of 
the fact that the 
geometric monodromy of $f$ fixes the boundary pointwise, in order to  relate 
relative cycles to their 
absolute 
counterparts. Historically, the operator has been extensively studied in the 
context 
of Picard--Lefschetz theory \cite{AGZVII}, and it is a classical result that in 
the case of isolated 
hypersurface singularities, it is a linear isomorphism. The starting motivation 
of this work is to understand the inverse of the variation operator, or rather, 
the inverse of a {\em geometric} version of it: a variation operator that takes 
properly 
embedded arcs to closed curves in the Milnor fiber.

We find that not all closed curves are in the image of a geometric variation 
operator
of a single properly embedded arc.
For example, we show that separating simple closed curves cannot be in the 
image.
To remedy this, we consider a finite disjoint collection of arcs, called an 
arcset.

A particularly interesting set contained in the collection of simple closed 
curves in $F$, is the set of vanishing cycles associated to $f$, that is, the 
curves that get contracted to a point in some nodal degeneration of $F$ in the 
versal unfolding space of $f$. In this work, we characterize which arcs are 
sent to vanishing cycles by the variation operator. The first main result of 
this work (\Cref{thm:single_arcs}) deals with the case of single arcs 
and gives a characterization purely in terms of intersection numbers of the arc 
and its image by the geometric monodromy: the image of a properly embedded arc 
by the geometric variation operator is a geometric vanishing cycle if and only 
if the arc and its image by the geometric monodromy  can be made disjoint in 
the interior of the Milnor fiber. The analogous result for the case of an 
arcset (\Cref{thm:vanishing_arcs_collection}) says that there are no 
obstruction 
for the image  to be a vanishing arc 
as long as the image is a simple closed curve.
 
 Now, we can ask a family version of this question. Namely, given  a 
 distinguished collection of vanishing cycles
 associated to a Morsification of $f$ with a choice of vanishing paths, we may 
 ask if each
 vanishing cycle is in the image of the geometric variation operator applied to 
 an arcset, and if
 there exists a collection of such arcsets with good properties.

We define a topological exceptional collection of vanishing arcsets. Like an 
exceptional collection in algebraic geometry, arcsets are
ordered and geometric monodromy image of the bigger arcsets  do not intersect 
the smaller arcsets.

For a totally real plane curve singularity $f$, A'Campo introduced the notion 
of a divide as a combinatorial tool where
the topology of the Milnor fiber and a distinguished collection of vanishing 
cycles can be 
read off.
In the second main result of the present work (\Cref{thm:maina}), we show that  
we can always find  topological exceptional collection of vanishing arcsets 
whose 
geometric
variation images are isotopic to the distinguished collection of vanishing 
cycles of A'Campo for any divide.

\subsection*{Vanishing arcs in symplectic geometry}

We comment on how our work relates to constructions in symplectic geometry.
The Milnor fiber of an isolated hypersurface singularity $f$ is known
to be a symplectic manifold.  Monodromy and vanishing cycles can be chosen to be an exact
symplectomorphism and  Lagrangian submanifolds respectively.

By selecting a Morsification of $f$ and a set of vanishing paths, the
directed Fukaya-Seidel category of an isolated singularity is defined from the
Lagrangian intersection theory of the distinguished collection of vanishing cycles. It was shown that
its derived category (or the $A_\infty$-triangulated envelope) is
independent of the choices involved.

The first four authors have recently constructed a categorical
analogue of variation operator in symplectic geometry. 
Namely, for any non-compact exact Lagrangian in the Milnor fiber, its
geometric variation image can be realized as compact exact
Lagrangian.

Furthermore, this information can be used to define a monodromy Fukaya
category of  an isolated hypersurface singularity $f$ (see \cite{BCCJ23}, \cite{BCCJ26}, see also \cite{CCJ20}).
One advantage of this construction is that it does not depend on the Morsification or the choice of vanishing paths.

It is conjectured that this monodromy Fukaya category is isomorphic to
the Fukaya-Seidel category for the case of two variables and that
the latter is embedded in the former in general. In this conjectural relation, 
distinguished collections of vanishing
cycles are expected to correspond to the exceptional collections.

A collection of an $A_{\infty}$ (or dg)-category is called {\em
exceptional} if self hom space is generated by identity and there
exist  no morphism from bigger to smaller indexed objects. We remark
that both collections admit braid group action.

By taking the Euler-characteristic of the conditions for an
exceptional collection of a $A_{\infty}$ (or dg)-category,  derives
the conditions of topological exceptional collections in this paper.
This paper suggests a refinement of the above conjectural relation. Namely, the exceptional
collection should consist of vanishing arcsets.

\subsection*{Organization of the paper}

This work is organized as follows. Section~\ref{sec:motivation} sets the stage 
by defining the variation operator and discussing its relevance in singularity 
theory. We also introduce the Seifert form which is later used to verify that 
the image of a single arc by the variation operator is a non-separating curve. 
Section~\ref{sec:winding_numbers} extends the theory of 
winding 
numbers to piecewise $C^1$ curves and arcs, laying the technical groundwork for 
proving the main results, we finish this sections with some interesting 
examples showing the existance of simple closed curves with vanishing winding 
number that are separating and thus, can't be vanishing cycles. In 
Section~\ref{sec:vanishing_cycles_arcs} explores the classical theory 
of vanishing cycles, introducing the concept of vanishing arcs as their 
relative counterparts.  
Section~\ref{sec:characterizing_arcs}, we 
characterize vanishing arcs using intersection numbers and geometric 
properties. 
Finally, Section~\ref{sec:finding_arcs} presents 
methods for 
constructing examples of vanishing arcs in the context of 
Brieskorn-Pham singularities.
In Section~\ref{sec:exceptcoll}, we introduce the notion of linear arcset and 
an exceptional collection of arcsets.
In Section~\ref{sec:divide}, we recall A'Campo's divide and its depth.
In Section~\ref{sec:depthzero}, we recall the notion of an adapted family which 
is useful to find the inverse image of the topological variation operator.
We find the exceptional collection of arcsets for depth zero cases.
In Section~\ref{sec:basic}, we define basic arcs corresponding to edges of 
A'Campo--Gusein-Zade diagram $A\Gamma(\mathbb{D}_{f})$.
In Section~\ref{sec:arcsetgeneral}, we define an arcset by collecting basic 
arcs along a good path in the diagram $A\Gamma(\mathbb{D}_{f})$.
We show that the chosen arcsets form topological exceptional collection of 
arcsets for a divide.

\section{Motivation} \label{sec:motivation}
Let $f: \C^2 \to \C$ be a complex analytic map that defines an isolated plane 
curve singularity at the origin. For $\epsilon>0$ small enough and $\delta>0$ 
sufficiently small with respect to $\epsilon$, the restriction of $f$
\[
f^{-1}(\partial \mathbb{D}_{\delta}) \cap \mathbb{B}_\epsilon \to \partial 
\mathbb{D}_\delta
\]
is a locally trivial fibration known as the {\em Milnor fibration}. We denote 
by $\Si_f$ one of its fibers and call it the {\em Milnor fiber}. The
characteristic mapping class of the Milnor fibration is called the {\em 
	geometric monodromy}. We denote this mapping class or a diffeomorphism 
representing it, by
\[
\varphi_f: \Si_f \to \Si_f.
\]
The hypothesis on $f$ defining an {\em isolated} plane curve singularity 
implies that $\varphi_f$ can be taken to be the identity on $\partial \Si_f$. 
In different words, (the class of) $\varphi_f$ is a well defined element of the 
relative mapping class 
group $\Mod(\Si_f)$ of diffeomorphisms of $\Si_f$ that fix the boundary 
pointwise up to isotopy preserving the action on the boundary.  Let 
$\left(\varphi_f\right)_\ast:H_1(\Si_f, \partial\Si_f; \ZZ)\to H_1(\Si_f, 
\partial\Si_f; \ZZ)$ be the map induced on relative homology by the geometric 
monodromy. We recall the definition of a classical operator.
\begin{definition}
	We define the {\em variation operator} $V_f$ associated with the isolated 
	plane 
	curve singularity $f$ by
	\begin{align*}
	V_f:H_1(\Si_f, \partial\Si_f; \ZZ) &\to H_1(\Si;\ZZ) \\
	 [a] &\mapsto \left[\varphi_f(a)-a\right]
	\end{align*}
	where $a$ is any relative cycle representing $[a]$.
\end{definition}
It is well defined because the boundary of the relative cycle 
$\varphi_f(a)$ coincides with the boundary of the relative 
cycle $a$. See \cref{fig:variation}.

\begin{figure}
	\centering
	\resizebox{0.2\textwidth}{!}{
\begingroup%
  \makeatletter%
  \providecommand\color[2][]{%
    \errmessage{(Inkscape) Color is used for the text in Inkscape, but the package 'color.sty' is not loaded}%
    \renewcommand\color[2][]{}%
  }%
  \providecommand\transparent[1]{%
    \errmessage{(Inkscape) Transparency is used (non-zero) for the text in Inkscape, but the package 'transparent.sty' is not loaded}%
    \renewcommand\transparent[1]{}%
  }%
  \providecommand\rotatebox[2]{#2}%
  \newcommand*\fsize{\dimexpr\f@size pt\relax}%
  \newcommand*\lineheight[1]{\fontsize{\fsize}{#1\fsize}\selectfont}%
  \ifx\svgwidth\undefined%
    \setlength{\unitlength}{116.28293015bp}%
    \ifx\svgscale\undefined%
      \relax%
    \else%
      \setlength{\unitlength}{\unitlength * \real{\svgscale}}%
    \fi%
  \else%
    \setlength{\unitlength}{\svgwidth}%
  \fi%
  \global\let\svgwidth\undefined%
  \global\let\svgscale\undefined%
  \makeatother%
  \begin{picture}(1,1.51929953)%
    \lineheight{1}%
    \setlength\tabcolsep{0pt}%
    \put(0,0){\includegraphics[width=\unitlength,page=1]{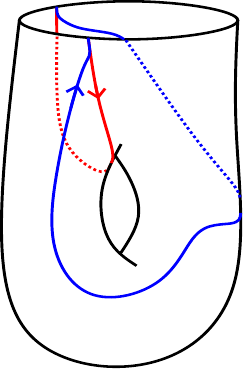}}%
    \put(0.49953438,1.00685425){\color[rgb]{0,0,0}\makebox(0,0)[lt]{\lineheight{1.25}\smash{\begin{tabular}[t]{l}\textbf{$a$}\end{tabular}}}}%
  \end{picture}%
\endgroup%
}
	\caption{An relative cycle (red) and its image with reversed orientation 
	(blue) by a diffeomorphism which is the identity on the boundary. }
	\label{fig:variation}
\end{figure}
The variation operator gives us a way to relate relative and absolute cycles 
but moreover, in the case of isolated hypersurface singularities, this operator 
is a linear isomorphism \cite[Theorem 2.2]{AGZVII}. It is important to remark 
that 
this is a theorem in singularity theory and that, in general, an analogous 
operator defined for a mapping class in $\Mod(\Si_f)$ does not yield an 
isomorphism. The proof of this result relies on Picard--Lefschetz theory and it 
is a further reflection of the fact that the monodromy of an isolated 
hypersurface singularity {\em moves everything around}. Other reflections of 
this 
phenomenon are, for example, the classical results that state the vanishing of 
the Lefschetz number $\Lambda_f=0$ of the monodromy \cite{NorLef}, or the more 
general result that there is a representative of the geometric 
monodromy that acts without fixed points \cite{Le}.

\subsubsection*{Relation with other invariants} Here we introduce other 
classical invariants that appear in the present paper and that are 
tightly related to the variation operator. Let $[c] \in H_1(\Si_f;\ZZ)$ be a 
cycle represented by a chain $c$ and let $\tilde{c}$ be the translation of $c$ 
to a nearby Milnor fiber in the positive direction indicated by the orientation 
of $\partial \mathbb{D}_\delta$, then the {\em Seifert form} is defined as
\[
\begin{split}
	\calL: H_1(\Si_f;\ZZ) \times  H_1(\Si_f;\ZZ)  &\to \ZZ \\
	\left([c],[d]\right) &\mapsto \mathrm{lk}(c,\tilde{d})
\end{split}
\]
that is, the linking number of $c$ and $\tilde{d}$ in the $3$-sphere. Note that 
even if we are using the Milnor fibration in the tube in this paper, it is 
equivalent to a fibration on the complement of a link in the $3$-sphere 
\cite{Milnor} and so this definition makes sense.
Finally, let $\bullet$ denote the {\em intersection pairing}
\[
\begin{split}
H_1(\Si_f,\partial \Si_f; \ZZ) \times H_1(\Si_f;\ZZ) &\to \ZZ \\
\left([a],[c]\right) &\mapsto [a] \bullet [c]
\end{split}
\]
that can be defined by taking the signed transversal intersection of a relative 
cycle representing $[a]$ and an absolute cycle representing $[c]$. The 
intersection 
pairing, which only depends on the topology of $\Si_f$ relates the variation 
operator and the Seifert form via \cite[Theorem 2.3]{AGZVII}
\begin{equation}\label{eq:var_seifert}
\calL([c],[c]) = \left(V_f^{-1}\left([c]\right)\right) \bullet [c],
\end{equation}
showing that $\Var_f$ and $\calL$ contain the same information. 

It is a consequence of the definition that knowing the monodromy {\em well 
enough} allows one to compute the variation operator. It is not so clear 
though, how to compute the inverse of the variation operator. Of course, it is 
always possible to compute the inverse of an integral matrix but one loses all 
geometrical aspect of the variation operator.

\subsection*{The geometric variation operator}
 In this subsection we define a geometric version of the variation operator 
 that takes arcs to closed curves. 

\subsubsection{Representing relative cycles}
The following lemma justifies the definition of the domain and target spaces of 
the geometric variation operator. A similar statement for absolute classes of 
homology and simple closed curves is true and more common in the literature 
(see for example \cite{Meeks} to understand the representation of primitive 
elements in absolute homology) but 
we couldn't find a reference for this relative counterpart.
\begin{lemma}\label{lem:representing_arcs}
	Let $\Si$ be an oriented compact surface with non-empty boundary. Then, 
	every relative class in $H_1(\Si,\partial \Si;\ZZ)$ can be 
	represented by a finite disjoint union of properly embedded arcs. 
\end{lemma}
\begin{proof}
	We have the identifications $
	H_1(\Si,\partial \Si;\ZZ) \simeq H^1(\Si; \ZZ) \simeq [\Si, \mathbb{S}^1]$,
	where the first isomorphism is Alexander duality and the second is basic 
	obstruction theory. Take an element $\alpha \in H_1(\Si,\partial \Si;\ZZ)$ 
	and 
	let 
	$\rho_\alpha\in  [\Si, \mathbb{S}^1]$ be the associated map by the above 
	identification. We can assume that $\rho_\alpha$ is smooth since every 
	continuous map between manifolds is homotopic to a smooth one (for a proof 
	of 
	this result, see for example
	\cite[Proposition 17.8]{Bott}). Let 
	$s\in\mathbb{S}^1$ be a regular value which exists by Sard's theorem. Then 
	$\rho_\alpha^{-1}(s) \subset \Si$ is a $1$-dimensional manifold 
	representing 
	$\alpha$. In particular, $\rho_\alpha^{-1}(s)$ is a finite disjoint union 
	of 
	simple closed curves and arcs.
	
	Finally, one can get rid of any simple closed curves. Let $\{c_1, \ldots, 
	c_k\} \subset \rho_\alpha^{-1}(s)$ be all the simple closed curves in 
	$\rho_\alpha^{-1}(s)$. Let $\hat{\Si} = \Si \setminus \bigcup_i c_i$. Let 
	$c_j$ be a curve corresponding to a boundary component of a connected 
	component of $\hat{\Si}$ that contains also a component of $\partial
	\Si$. By conjugating the curve $c_j$ by a path from a point in $c_j$ 
	to a 
	point 
	on $\partial \Si$, one turns the curve $c_j$ into an arc that represents 
	the same class in relative homology. Repeat this process until one has got 
	rid of all curves in $\rho_\alpha^{-1}(s)$.
\end{proof}

\subsection*{The geometric variation operator}
For a surface $\Si$, let $\calC_\Si$ be the set of piecewise 
$C^1$ closed curves (possibly not simple). And let $\calI_\Si$ be the set of 
piecewise $C^1$ properly 
embedded arcs. That is,  the elements of  $\calC_\Si$ and $\calI_\Si$ are 
concatenations $a_1 \ast \cdots \ast a_k$ of arcs $a_i:I_i \to \Si_f$ (where 
$I_i$ is a connected closed segment) which are $C^1$ embeddings. So the end 
point of $a_i$ coincides with the starting 
point of $a_{i+1}$. Furthermore, in the case of closed curves the endpoint of 
$a_k$ is the starting point of $a_1$ and, in the case of properly embedded 
arcs, the starting point of $a_1$ and the endpoint of $a_k$  lie on $\partial 
\Si$ and the arcs are transverse to $\partial \Si$ at those points.

\begin{definition}
 We define the {\em geometric variation operator on single arcs} associated 
 with $\varphi_f$ as the map
\begin{align*}
	\Var_f:\calI_{\Si_f} &\to \calC_{\Si_f} \\
	a & \mapsto (\varphi_f(a))\ast  (-a)
\end{align*}
where $\phi_f(a)$ is the composition of the arc with the geometric monodromy 
and where $(-a)(t)=a(1-t)$.
\end{definition}
\begin{remark}
	The geometric variation operator $\Var_f$ induces the classical 
	variation operator $V_f$ for classes that can be represented by single arcs.
	
	Note that the image of a properly embedded arc can be a closed 
	curve with non-vanishing self intersection number.
\end{remark}

\subsection*{The Seifert form on separating curves} In this subsection we study 
the action of the Seifert form on separating curves. We deduce numerical 
constraints from a work of R. Kaenders \cite{Kaen} which in turn is strongly 
based on a previous work by E. Selling \cite{Sell} on quadratic forms.

\begin{lemma}\label{lem:seifert_separating}
	Let $f$ define an isolated plane curve singularity other than $A_1$ and let 
	$c \subset \Si_f$ be 
	a 
	non-nullhomologous separating simple closed 
	curve, then
	\[
	\calL([c],[c]) \leq -2.
	\] 
\end{lemma}
\begin{proof}
	Let $r$ be the number of branches of the plane curve singularity defined by 
	$f$. That $c$ is separating with $[c] \neq 0$ in homology, implies that $r 
	\geq 
	2$, that is, that $f$ has at least two branches.
	
	Let $\Delta_1, \ldots,  \Delta_r$ be the $r$ boundary 
	components of $\Si_f$ with the orientation inherited from $\Si$ so 
	that $\sum_i[\Delta_i] = 0$ holds in homology.  The radical of the 
	intersection 
	form $\calS$ of 
	$\Si_f$ is 
	generated by the classes of $\Delta_1, \ldots,  \Delta_r$ (see \cite{Kaen} 
	for more 
	about this). By hypothesis, the 
	curve $c$ splits $\Si_f$ in two components $\Sigma_1$ and $\Sigma_2$. 
	Orient $c$ as a boundary component of $\Sigma_2$. Then 
	\[
	[c] = \sum_i \delta_i [\Delta_i].
	\]
	where $\delta_i=1$ if $\Delta_i$ is a boundary component of $\Sigma_1$ and 
	$\delta_i=0$ if $\Delta_i$ is a boundary component of $\Sigma_2$.
	Using the formula in \cite[Proposition 2.2]{Kaen}, we find that
	\[
	\calL([c],[c]) = \sum_{1 \leq i < j \leq r} -\nu_{ij}(\delta_i-\delta_j)^2.
	\]
	This sum contains non-zero terms because $[c] \neq 0$ implies that there 
	are boundary components of $\Si_f$ on both sides of $c$. Note that 
	$\nu_{ij} \geq 1$ and it is exactly $1$ only when $\Delta_i$ and $\Delta_j$ 
	form an $A_1$ singularity, that is, when they are smooth transversal 
	branches meeting at a point. Since by hypothesis $f$ is not an $A_1$ 
	singularity, then, either $r=2$ and $\nu_{12}\geq 2$; or $r > 2$ and there 
	are at least two non-zero terms in the above sum, proving 
	the result.
\end{proof}

\begin{lemma}\label{lem:inverse_separating}
	Let $f$ define an isolated plane curve singularity. Let $c \subset \Si_f$ 
	be a  non-nullhomologous separating simple closed 
	curve. Then, $V_f^{-1}([c])$ can't be represented by a single properly 
	embedded arc.
\end{lemma}
\begin{proof}
 By the previous \Cref{lem:seifert_separating} and by \cref{eq:var_seifert}, we 
 have the inequality 
$V_f^{-1}([c]) \bullet [c] 
\leq -2$. But the algebraic intersection number of a separating curve and a 
single properly embedded arc is either $-1$, $1$ or $0$ depending on the 
orientation of the arc and on whether both ends of the arc lie on the 
same or different components of $ \Si_f \setminus \{c\}$.
\end{proof}

The previous lemmas are used at the end of the following section to produce an
interesting example that shows the necessity of certain hypothesis in our 
theorems. But, before we are able to explain it, we need to introduce {\em 
winding numbers}.

\section{Winding numbers of curves and arcs}
\label{sec:winding_numbers}
Let $\Si$ be an oriented compact surface with non-empty boundary. Let 
$\calI_\Si$ be the set of piecewise $C^1$ properly embedded arcs of the surface 
$\Si$. And let $\calC_S$ be the set of piecewise $C^1$ closed curves of $S$ 
with  possibly  self intersections (so not necessarily simple closed curves).

\subsection*{Relative framings and relative winding number functions}
In this subsection we recall definitions and properties of {\em relative 
	framings} of a surface (see also 
\cite[Section 2]{CalSal}). 

A {\em framing} of $\Si$ is a 
trivialization of the tangent bundle $T\Si$. With a Riemannian metric fixed, 
framings of $\Si$ are in correspondence with nowhere vanishing vector fields on 
$\Si$, or equivalently, with isomorphisms of $SO(2)$ bundles 
$\mathbb{S}^1(T\Si) \simeq \Si 
\times \mathbb{S}^1$ where $\mathbb{S}^1(T\Si) $ is the circle tangent bundle 
of $\Si$.

Two framings $\phi$ and $\psi$ are {\em isotopic} if the corresponding 
vector fields are isotopic through non-vanishing vector fields, and are {\em 
	relatively isotopic} if the isotopy can be chosen to act trivially on
$\partial \Si$. 

Let $\phi_\xi$ be a framing corresponding with a nowhere vanishing vector field 
$\xi$ and let $\gamma: [0,1]\to \Si$  be a $C^1$ embedding with 
$\gamma(0)=\gamma(1)$ and $\gamma'(0)=\gamma'(1)$. Equivalently, $\gamma$ is a 
representative of a $C^1$ simple closed curve. Given such piece of that, we 
can associated to $\gamma$ an integer which is called the {\em winding 
	number}. This measures how the vector 
field $\xi_\phi|_{\gamma(t)}$ {\em winds} around the
forward-pointing vector field $\gamma'(t)$.
\begin{equation}\label{eq:winding_embedding}
	\phi_\xi(\gamma) = \int_{0}^1 d \,\ang \left(\gamma'(t), 
	\xi_{\gamma(t)}\right) 
	\in \ZZ.
\end{equation}
The integer $\phi_\xi(\gamma)$ is invariant under isotopy of both 
$\xi$ and $\gamma$. Letting $\calC^1_\Si$ denote the set of isotopy 
classes 
of oriented simple closed curves of $\Si$ defined by $C^1$ embeddings. Then, 
\cref{eq:winding_embedding} defines a map
\[
\phi_\xi: \calC^1_\Si \to \ZZ.
\]

Suppose now that each boundary component $\Delta_i$ of $\Si$ is equipped with a 
point $p_i$ such that $\xi$ is inward-pointing at $p_i$. We call such 
$p_i$ a {\em legal basepoint}.
Choose exactly one legal basepoint on each boundary component. One 
might be concerned about the possibility that no legal basepoints exist, but 
this only happens when the boundary component has zero winding 
number. As was 
shown in \cite{Sal} using \cite{Khan}, boundary components of Milnor fibers 
equipped 
with the complex Hamiltonian vector field 
$
\xi_f = \frac{\partial f}{\partial y} \frac{\partial }{\partial x} - 
\frac{\partial f}{\partial x} \frac{\partial }{\partial y} = 
\left(\frac{\partial f}{\partial y},- 
\frac{\partial f}{\partial x}\right)
$
have negative winding number. A {\em legal 
	arc} on $\Si$ is a 
properly-embedded arc $a: [0,1] \to \Si$ that begins and ends at distinct 
legal basepoints, and such that $a$ is tangent to $\xi$ at both 
endpoints. The winding number of a legal arc is then necessarily of the form $c 
+ \frac{1}{2}$ for $c \in \ZZ$, and is invariant up to isotopy through legal 
arcs. Observe also that $\Mod(\Si)$ acts on 
the set of legal isotopy classes of legal arcs. 

We let $\calC_\Si^{1,+}$ be the set obtained from $\calC^1_\Si$ by adding all 
isotopy classes of oriented legal arcs. Having chosen a system 
of legal 
basepoints, a framing $\phi_\xi$ gives rise to a {\em relative winding number 
	function}
\[
\phi: \calC_\Si^{1,+} \to \tfrac{1}{2} \ZZ.
\]
The relative winding number function associated to a framing $\phi$ is clearly 
invariant under relative isotopies of the framing. Crucially, the converse 
holds as well. 

\begin{prop}[c.f. Proposition 2.1, \cite{CalSal}]\label{prop:relative_wn}
	Let $\Si$ be a surface of genus $g \ge 2$, and let $\phi$ and $\psi$ be 
	framings of $\Si$ that restrict to the same framing of $\partial \Si$. If 
	the 
	relative winding number functions associated to $\phi$ and $\psi$ are 
	equal, then the framings $\phi$ and $\psi$ are relatively isotopic. 
\end{prop}

\begin{remark}[Good arcs]\label{rem:invariance_rlwn}
	Observe that the restriction of choosing {\em exactly one legal base point} 
	at 
	each boundary component highlights the strength of \Cref{prop:relative_wn} 
	since it says that it is only necessary to check the values of two relative 
	winding number functions on simple closed curves and {\em legal arcs} to 
	verify 
	if the corresponding vector fields are isotopic. However, in order to have 
	a 
	well-defined winding number function we can consider, and will do so from 
	now 
	on,  what we call {\em good
		arcs}. 
	
	A {\em good arc} is a properly embedded arc $a:[0,1] \to \Si$, transverse 
	to 
	$\partial \Si$, with the 
	property that $a'(0)= \pm k \xi_f(a(0))$ and $a'(1)= \mp k' \xi_f(a(1))$ 
	with 
	$k,k' \in \RR_{>0}$. This property guarantees that $\phi (a)$ is of the 
	form 
	$c+1/2$ with $c \in \ZZ$ just like in the case of legal arcs. Note that 
	with 
	this definition we allow a lot more of arcs since, for example,  we allow 
	them 
	to start 
	and end at the same boundary component which was not allowed in the 
	definition 
	of legal arc because starting and endpoints were required to be distinct.
	
	Observe also that, we can always require $\varphi \in \Mod(\Si)$ to be the 
	identity on a small collar neighborhood of $\partial \Si$, and so 
	$\Mod(\Si)$ 
	acts on 
	the set of isotopy classes of goods arcs where isotopies are required to be 
	along good arcs. 
\end{remark}
\subsection*{Winding numbers for piecewise $C^1$ curves and arcs} 
\label{sec:piecewise_c1}For the 
purposes of this work, it is necessary to extend the definition of winding 
number beyond the case of $C^1$ embeddings of curves and arcs. We note that 
there is a natural theoretical generalization for $C^0$ embeddings. Indeed, let 
$c: 
\mathbb{S}^1 \to \Si$ be a topological embedding of a circle into $\Si$. It is 
a classical theorem in the theory of mapping 
class groups that $c$ can be approximated by $C^1$ (or even smooth) simple 
closed curves and that, even more, curves that are close enough in some 
appropriate compact-open topology, are actually isotopic to $c$ (see 
\cite[1.2.2]{Farb}). Also, any two 
approximations $c'$ and $c''$ to $c$ are themselves isotopic through $C^1$ 
embeddings. Thus, there is a way of defining a winding number for any 
topologically embedded simple closed curve. In \cite{ChillI, ChillII}, with 
more effort, the notion of winding number is even extended to all $\pi_1(\Si)$ 
and coincides with the description we just gave for topological embeddings of 
$\mathbb{S}^1$.

However, this approach is not very friendly from a calculation point of view. 
That is, it is in general not very manageable to deal with approximations and 
the process misses the practical point of winding numbers. So the way we tackle 
this issue in this work in somewhat intermediate. We deal with the case of 
immersed piecewise 
$C^1$ simple closed curves and arcs which are defined by concatenated 
immersions of intervals (see \cite{Rein} for more on this approach). Let us 
define the winding number $\phi_\gamma$ of any immersed $C^1$ 
arc $a$ by the 
\cref{eq:winding_embedding} so $\phi_\gamma(a) \in \RR$ and is no longer an 
integer.
Let $a$ be either a simple closed curve or a properly embedded arc which is 
defined by a
concatenation of properly embedded $C^1$ arcs:
\[
a=a_1 \ast \cdots \ast a_k.
\]
And let $\theta_j = \ang \left(a_{j}'(1), a_{j+1}'(0)\right)$ for $j= 1, 
\ldots, k-1$ and let $\theta_k = \ang \left(a_{k}'(1), a_{1}'(0)\right)$ if 
$a_{k}(1) =a_{1}(0)$ and $\theta_k =0$ otherwise. Assume for the moment that 
$|\theta_j| < \pi$ for all $j \in \{1, \ldots, k\}$. In this case, we can define
\begin{equation}\label{eq:def_C1}
	\phi_\xi (a) = \sum_{j=1}^k \phi_\xi (a_j) + \theta_j \in \ZZ.
\end{equation}

\begin{figure}
	\centering
	\small
	\resizebox{0.9\textwidth}{!}{
\begingroup%
  \makeatletter%
  \providecommand\color[2][]{%
    \errmessage{(Inkscape) Color is used for the text in Inkscape, but the package 'color.sty' is not loaded}%
    \renewcommand\color[2][]{}%
  }%
  \providecommand\transparent[1]{%
    \errmessage{(Inkscape) Transparency is used (non-zero) for the text in Inkscape, but the package 'transparent.sty' is not loaded}%
    \renewcommand\transparent[1]{}%
  }%
  \providecommand\rotatebox[2]{#2}%
  \newcommand*\fsize{\dimexpr\f@size pt\relax}%
  \newcommand*\lineheight[1]{\fontsize{\fsize}{#1\fsize}\selectfont}%
  \ifx\svgwidth\undefined%
    \setlength{\unitlength}{437.32843835bp}%
    \ifx\svgscale\undefined%
      \relax%
    \else%
      \setlength{\unitlength}{\unitlength * \real{\svgscale}}%
    \fi%
  \else%
    \setlength{\unitlength}{\svgwidth}%
  \fi%
  \global\let\svgwidth\undefined%
  \global\let\svgscale\undefined%
  \makeatother%
  \begin{picture}(1,0.33224587)%
    \lineheight{1}%
    \setlength\tabcolsep{0pt}%
    \put(0,0){\includegraphics[width=\unitlength,page=1]{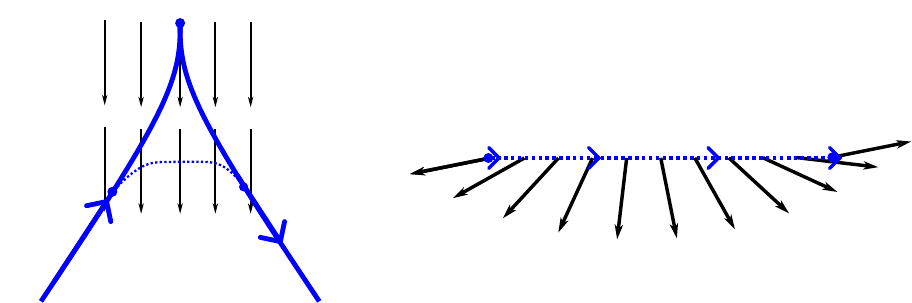}}%
    \put(0.18264975,0.32220441){\color[rgb]{0,0,0}\makebox(0,0)[lt]{\lineheight{1.25}\smash{\begin{tabular}[t]{l}\textbf{$c_j$}\end{tabular}}}}%
    \put(0.00031048,0.11800069){\color[rgb]{0,0,0}\makebox(0,0)[lt]{\lineheight{1.25}\smash{\begin{tabular}[t]{l}\textbf{$a_j(1-1/n)$}\end{tabular}}}}%
    \put(0.20186316,0.16460602){\color[rgb]{0,0,0}\makebox(0,0)[lt]{\lineheight{1.25}\smash{\begin{tabular}[t]{l}\textbf{$b_n$}\end{tabular}}}}%
    \put(0.69426002,0.17692462){\color[rgb]{0,0,0}\makebox(0,0)[lt]{\lineheight{1.25}\smash{\begin{tabular}[t]{l}\textbf{$b_n$}\end{tabular}}}}%
    \put(0.52905733,0.18013665){\color[rgb]{0,0,0}\makebox(0,0)[lt]{\lineheight{1.25}\smash{\begin{tabular}[t]{l}\textbf{$a_j(1-1/n)$}\end{tabular}}}}%
    \put(0.28046787,0.12406459){\color[rgb]{0,0,0}\makebox(0,0)[lt]{\lineheight{1.25}\smash{\begin{tabular}[t]{l}\textbf{$a_{j+1}(1/n)$}\end{tabular}}}}%
    \put(0.90125439,0.18552994){\color[rgb]{0,0,0}\makebox(0,0)[lt]{\lineheight{1.25}\smash{\begin{tabular}[t]{l}\textbf{$a_{j+1}(1/n)$}\end{tabular}}}}%
  \end{picture}%
\endgroup%
}
	\caption{On the left we see the chart $U$. The points $a_{j}(1-1/n),c_j$ 
		and $a_{j+1}(1/n)$ are marked. The arc $b_n$ between $a_{j}(1-1/n)$ 
		and $a_{j+1}(1/n)$ is dotted in blue. In black we see the vector field 
		$\xi$ which in this case is tangent to the interval $a_j$ and $a_{j+1}$ 
		at 
		$c_j$}
	\label{fig:approx_C1_arc}
\end{figure}
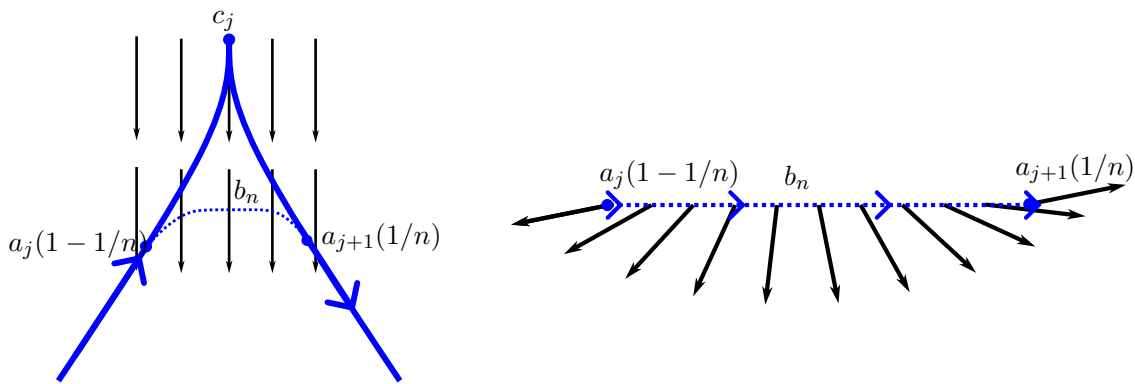

This leaves out an important case for this paper: when 
$\theta_j = \pm \pi$. This situation, which is also not covered in \cite{Rein}, 
plays an important role here. In 
this case we do the following in order to the 
decide the correct sign of $\theta_j$ so that the formula from \cref{eq:def_C1} 
is still valid. Let 
$c_j=a_j(1)=a_{j+1}(0)$ be the intersection point of two consecutive segments 
of $a$, and let $\rho:U \to \RR^2$ be a small chart of $\Si$ 
around $c_j$. See \cref{fig:approx_C1_arc} to follow this construction in the 
important case when all three vectors $a_j'(1),a_{j+1}'(0)$ and 
$\xi_{a_j(1)}$ lie on the same line (in particular $\theta_j = \pm \pi$) :
\begin{enumerate}
	\item let $p_n = a_j(1-1/n)$ and let $q_n=a_{j+1}(1/n)$ for $n \in 
	\ZZ_{>0}$ be sequences of points in both segments converging to $c_j$
	\item for $n$ big enough, $p_n$ and $q_n$ are in $U$.  We define $b_n:[0,1] 
	\to U$ as a smooth arc satisfying
	\begin{enumerate}
		\item $b_n(0) = a_j(1-1/n)$ and $b_n(1)=a_{j+1}(1/n)$,
		\item $b_n'(0) = a_j'(1-1/n)$ and $b_n'(1)=a_{j+1}'(1/n)$,
		\item $b_n(t) = (1-t)a_j(1-1/(n+1)) + ta_{j+1}(1/(n+1))$, for $t \in 
		(\epsilon, 1-\epsilon)$. That is, on the interval $(\epsilon, 
		1-\epsilon)$, the curve $b_n$ is the linear interpolation between 
		$a_j(1-1/(n+1))$ and $a_{j+1}(1/(n+1))$,
		\item for $t \in [0, \epsilon]$, the curve $b_n(t)$ is any $C^1$ curve  
		that satisfies 
		\[b'_n(t)=
		(1-t/\epsilon)a_j'(1-1/n) + t/\epsilon \left(a_{j+1}(1/(n+1)) - 
		a_j(1-1/(n+1)) 
		\right),\]
		\item and similarly, for $t \in [1-\epsilon, 1]$, the curve $b_n(t)$ is 
		any 
		$C^1$ curve whose forward-pointing vector interpolates linearly between 
		the 
		vector $(a_{j+1}(1/(n+1)) - a_j(1-1/(n+1))$ and $a_{j+1}'(1/n)$. 
	\end{enumerate} 
	\item for $n$ big enough, the concatenation $a_j([0, 1-1/n]) \ast b_n \ast 
	a_{j+1}([0,1/n]) \cap U$ is homotopic to $a_j \ast a_{j+1} \cap U$.
	\item by the flowbox theorem, since $\xi$ has no singular points, for $n$ 
	big enough, $I_n=b_n([0,1])$ is small enough and $\xi|_{I_n}$ is 
	transversal to $I_n$ 
	and points always either to the right of $I_n$ or to the left of $I_n$ with 
	$I_n$ oriented by the forward-pointing vector of $b_n$. See 
	right hand side of \cref{fig:approx_C1_arc}.
	
	\item if $\xi|_{I_n}$ points to the right, then 
	$\pi<\ang\left(a'_j(1),b_n'(0)\right) 
	< 0$ and also $\pi<\ang\left(b_n'(1), a'_{j+1}(0)\right) < 0$. The 
	inequalities are inverted if $\xi|_{I_n}$ points to the left.
	\item $\theta_j = \pi$ if 
	$\xi$ points to the left of $I_n$ or 
	equivalently, if the forward-pointing vector of $I_n$ at $p$ and $\xi_p$ 
	(in that order) form a positive bases of the tangent plane $T_p\Si$. We 
	define $\theta_j = -\pi$ in the other case.
\end{enumerate}

Our choice of signs gives, of course, the same values as if we approximated  
our 
curve by the small arcs $b_n$ near the conflicting points. And that is 
precisely the proof that our choice 
agrees with the theoretical way of assigning winding numbers to every $C^0$ 
curve.

\begin{remark}\label{rem:good_arc_not_imp}
	Let $a: [0,1] \to \Si_f$ be any properly embedded arc, by isotoping $a$ 
	possibly sliding it along $\partial \Si_f$ along properly embedded arcs, we 
	can take $a$ to a good arc $a'$. In the process, we find two arcs $b_0$ and 
	$b_1$ such that $b_0 \ast a' \ast b_1$ is an arc which is relatively 
	isotopic to $a$. Assume that $\Var_{f}(a)$ is a simple closed curve. Then,
	\[
	\begin{split}
		\phi_f(\Var_{f}(a)) &= \phi_f(\varphi_f(a))-\phi_f(a) \\&= 
		\phi_f(\varphi_f(b_0 
		\ast a' 
		\ast b_1))-\phi_f(b_0 \ast a' \ast b_1) \\
		&= \phi_f(\varphi_f(a'))-\phi_f(a') \\
		&= \phi_f(\Var_{f}(a')).	
	\end{split}
	\]
	Which shows that, for the purposes of this work, one does not have to worry 
	about the behaviour of the arcs near the boundary but, rather legal and 
	good 
	arcs are a
	technical tool in order to have a convenient codomain for relative winding 
	numbers functions and prove results like \Cref{prop:relative_wn}.
\end{remark}

Next, we explain the promised example at the end of the previous section.
\subsection*{An example}
The following example shows the existence, in Milnor fibers of isolated plane 
curve singularities, of separating simple closed curves with winding number 
equal $0$. But being separating prevents them from being geometric vanishing 
cycles. That vanishing cycles are non-separating simple closed curves follows, 
for example, from the connectivity of the Dynkin diagram where algebraic 
intersection numbers are considered to draw edges (see \cite{Gab}) and the fact 
that separating curves have $0$ algebraic intersection number with any other 
curve.

\begin{example}\label{ex:separating_wn0}
	Let $f(x,y) = (y^3-x^4)x$. Since $f$ defines a plane curve singularity $B$ 
	with two branches $B_1$ and $B_2$ at the origin. The Milnor fiber $\Si_f$ 
	has two boundary components $\Delta_1$ and $\Delta_2$. 
		\begin{figure}
		\centering
		\resizebox{0.8\textwidth}{!}{
\begingroup%
  \makeatletter%
  \providecommand\color[2][]{%
    \errmessage{(Inkscape) Color is used for the text in Inkscape, but the package 'color.sty' is not loaded}%
    \renewcommand\color[2][]{}%
  }%
  \providecommand\transparent[1]{%
    \errmessage{(Inkscape) Transparency is used (non-zero) for the text in Inkscape, but the package 'transparent.sty' is not loaded}%
    \renewcommand\transparent[1]{}%
  }%
  \providecommand\rotatebox[2]{#2}%
  \newcommand*\fsize{\dimexpr\f@size pt\relax}%
  \newcommand*\lineheight[1]{\fontsize{\fsize}{#1\fsize}\selectfont}%
  \ifx\svgwidth\undefined%
    \setlength{\unitlength}{533.47256398bp}%
    \ifx\svgscale\undefined%
      \relax%
    \else%
      \setlength{\unitlength}{\unitlength * \real{\svgscale}}%
    \fi%
  \else%
    \setlength{\unitlength}{\svgwidth}%
  \fi%
  \global\let\svgwidth\undefined%
  \global\let\svgscale\undefined%
  \makeatother%
  \begin{picture}(1,0.29129786)%
    \lineheight{1}%
    \setlength\tabcolsep{0pt}%
    \put(0,0){\includegraphics[width=\unitlength,page=1]{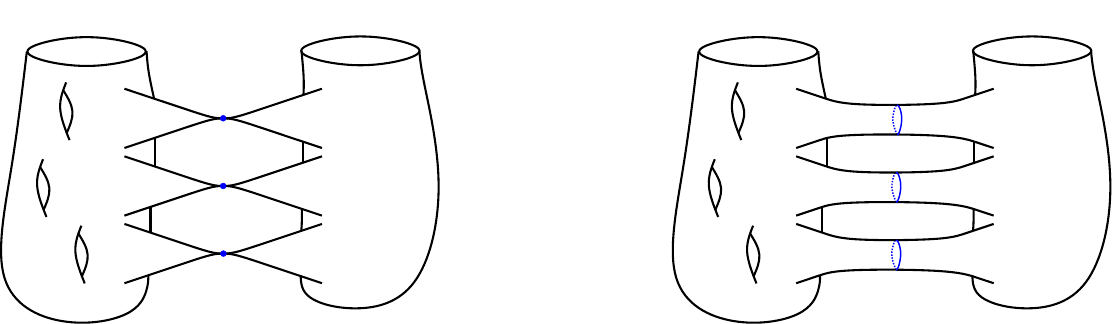}}%
    \put(0.039301,0.27778782){\color[rgb]{0,0,0}\makebox(0,0)[lt]{\lineheight{1.25}\smash{\begin{tabular}[t]{l}\textbf{$\Delta_1$}\end{tabular}}}}%
    \put(0.29651043,0.27778782){\color[rgb]{0,0,0}\makebox(0,0)[lt]{\lineheight{1.25}\smash{\begin{tabular}[t]{l}\textbf{$\Delta_2$}\end{tabular}}}}%
    \put(0.65089102,0.27726469){\color[rgb]{0,0,0}\makebox(0,0)[lt]{\lineheight{1.25}\smash{\begin{tabular}[t]{l}\textbf{$\Delta_1$}\end{tabular}}}}%
    \put(0.90810049,0.27726469){\color[rgb]{0,0,0}\makebox(0,0)[lt]{\lineheight{1.25}\smash{\begin{tabular}[t]{l}\textbf{$\Delta_2$}\end{tabular}}}}%
  \end{picture}%
\endgroup%
}
		\caption{On the left we see the nodal curve $\Si_1 \cup \Si_2$ and in 
			blue we see the three points $\Si_1 \cap \Si_2$. On the right we 
			see 
			the result after smoothing out the three $A_1$ points. This surface 
			is 
			homeomorphic to the Milnor fiber $\Si_f$.}
		\label{fig:example_van_sep}
	\end{figure}
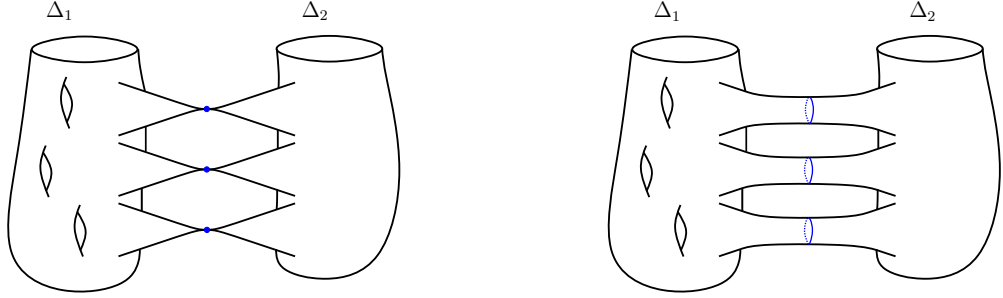
	Next, we do a construction to compute the winding numbers 
	$\varphi_f(\Delta_1)$ 
	and $\varphi_f(\Delta_2)$.
	Let $f_t(x,y) = (y^3-x^4-t)(x-t)$ be a deformation of $f$. For $t \neq 0$ 
	and small, the curve $B_t = f_t^{-1}(0) \cap \mathbb{B}_\epsilon$ defined 
	by $f_t$ is a nodal curve consisting of the Milnor fiber $\Si_1$ of $B_1$ 
	meeting transversely the Milnor fiber $\Si_2$ of $B_2$ in $3$ points 
	(because $3$ is the intersection multiplicity $B_1\cdot B_2$ of the two 
	branches). This construction shows that the Milnor fiber $\Si_f$ can be 
	constructed by performing a triple connected sum of $\Si_1$ and $\Si_2$, 
	or equivalently up to homeomorphism, by removing $3$ disks from each Milnor 
	fiber and gluing the boundary components using $3$ cylinders as in 
	\cref{fig:example_van_sep}. Moreover, the core curves of each of these 
	cylinders are geometric vanishing cycles.

	Therefore, using \cite[Theorem B]{Sal} and the Homological Coherence 
	Property (see 
	\cite[Lemma 2.4]{CalSal} or \cite{Humph}) 
	we find that
	\[
	\begin{split}
	 \phi_f(\Delta_1) &= \chi(\hat\Si_1) = 2-2g(\Si_1) - 4 = 2-2*3-4=-8 \\
	\phi_f(\Delta_2) &= \chi(\hat\Si_2) = 2-2g(\Si_2) - 4 = 2-2*0-4=-2.
	\end{split}
	\]
	Where we are using that the three vanishing cycles of 
	\cref{fig:example_van_sep} and $\Delta_1$ bound a 
	surface of genus $g(\Si_1)$ in the first case. And analogously for the 
	second case.
	
	Again, using homological coherence, we find that any separating simple 
	closed curve $c$ that separates $\Si_f$ into a surface of genus $4$ that 
	contains $\Delta_1$ and a surface of genus $1$ containing $\Delta_2$ is a 
	simple 
	closed curve with $\phi_f(c)=0$. 
	
\begin{figure}
		\centering
		\resizebox{0.5\textwidth}{!}{
\begingroup%
  \makeatletter%
  \providecommand\color[2][]{%
    \errmessage{(Inkscape) Color is used for the text in Inkscape, but the package 'color.sty' is not loaded}%
    \renewcommand\color[2][]{}%
  }%
  \providecommand\transparent[1]{%
    \errmessage{(Inkscape) Transparency is used (non-zero) for the text in Inkscape, but the package 'transparent.sty' is not loaded}%
    \renewcommand\transparent[1]{}%
  }%
  \providecommand\rotatebox[2]{#2}%
  \newcommand*\fsize{\dimexpr\f@size pt\relax}%
  \newcommand*\lineheight[1]{\fontsize{\fsize}{#1\fsize}\selectfont}%
  \ifx\svgwidth\undefined%
    \setlength{\unitlength}{123.53462039bp}%
    \ifx\svgscale\undefined%
      \relax%
    \else%
      \setlength{\unitlength}{\unitlength * \real{\svgscale}}%
    \fi%
  \else%
    \setlength{\unitlength}{\svgwidth}%
  \fi%
  \global\let\svgwidth\undefined%
  \global\let\svgscale\undefined%
  \makeatother%
  \begin{picture}(1,0.32431511)%
    \lineheight{1}%
    \setlength\tabcolsep{0pt}%
    \put(0,0){\includegraphics[width=\unitlength,page=1]{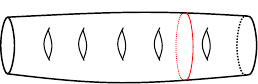}}%
    \put(0.69382912,0.29916396){\color[rgb]{0,0,0}\makebox(0,0)[lt]{\lineheight{1.25}\smash{\begin{tabular}[t]{l}\textbf{$c$}\end{tabular}}}}%
  \end{picture}%
\endgroup%
}
		\caption{In red, a separating simple closed curve in $\Si_f$ which, by 
		the homological coherence property, has vanishing winding number.}
		\label{fig:sep_4_1}
	\end{figure}
	
\end{example}

\begin{remark}\label{rem:not_bijective}
The previous example shows that the geometric variation operator $\Var_f$ is 
not a bijection. This produces a contrast with the classical theorem  
\cite[Theorem 2.2]{AGZVII} that 
$V_f$ is an isomorphism. Observe that $[c] \neq 0$ in absolute homology and 
so $V_f^{-1}([c]) \neq 0$ in relative homology. In particular by 
\Cref{lem:representing_arcs}, it is possible to represent $V_f^{-1}([c])$ by 
a disjoint union of properly embedded arcs. The previous 
\Cref{ex:separating_wn0} 
shows that one must use at least $2$ arcs. 

A similar version of this phenomenon is observed in the last section of 
\cite{BCCJ23} where {\em difficulties} are found for a geometric vanishing 
cycle 
to be 
the image of a single arc by the variation operator.
\end{remark}

\section{Vanishing cycles and arcs}
\label{sec:vanishing_cycles_arcs}
In this section, we recall the necessary definitions of versal unfolding to 
properly introduce algebraic and geometric vanishing cycles. We compare these 
two notions  via \Cref{lem:geom_finer} showing that geometric vanishing cycles 
contain strictly more information.

Finally, we introduce the notion of {\em vanishing arcs} (and its geometric 
version) as the counterpart in 
relative homology of a vanishing cycle.

\subsection*{The versal deformation space and the geometric monodromy 
	group}\label{subsection:versal}

\subsubsection*{The versal unfolding} We briefly recall here the notion 
of the 
{\em versal unfolding} of an isolated singularity; see \cite[Chapter 
3]{AGZVII} for more details.
Let $g_1, \dots, g_\mu \in \C[x,y]$
be polynomials that project to a basis of $A_f = \frac{\C\{x,y\}}{(\partial f/ 
\partial x, \partial f/\partial y)}$, assume 
that $g_1=1$. For $\lambda = 
(\lambda_1, \dots, \lambda_\mu) \in \C^\mu$, define the function $f_\lambda$ by
\[
f_\lambda = f + \sum_{i = 1}^\mu \lambda_i g_i.
\]
The {\em base space of the versal unfolding } of $f$ is the parameter space of 
all $\lambda$ which is naturally isomorphic to $\C^\mu$. The {\em discriminant 
	locus} is the 
subset
\[
\mbox{Disc}  = \{\lambda \in \C^\mu \mid f_\lambda^{-1}(0) \mbox{ is not 
	smooth}\}.
\]
The discriminant $\mbox{Disc}$ is an irreducible algebraic hypersurface.  The 
smooth part of $\mbox{Disc}$
parametrizes curves with a single node.  Denote by $\mathbb{B}_f$ a small 
closed ball 
in $\C^\mu$ centered at the origin.  Define
\begin{equation}\label{eq:taut}
	X_f = \{(\lambda, (x,y)) \mid (x,y) \in f_\lambda^{-1}(0),\ \lambda \not 
	\in \mbox{Disc}\}.
\end{equation}
Then, for $\mathbb{B}_f$ small enough and after intersecting $X_f$  with a 
sufficiently small closed polydisk, this family has the structure of a smooth 
surface bundle with base $\mathbb{B}_f \setminus \mbox{Disc}$ and fibers 
diffeomorphic 
to the Milnor fiber $\Si_f$ of the Milnor fibration.
We fix a point in $\mathbb{B}_f \setminus \mbox{Disc}$ and we denote, also by 
$\Si_f$, 
the fiber with boundary lying over it.

\begin{definition}\label{def:geom_monodromy_group}
	The {\em geometric monodromy group} is the image in $\Mod(\Sigma_f)$ of the 
	monodromy representation $
	\rho: \mathbb{B}_f \setminus \mbox{Disc} \to \Mod(\Si_f)
	$ of the universal family $X_f$ of \cref{eq:taut}.
\end{definition}

\begin{definition}\label{def:vanishing_cycle}
	A  {\em geometric  vanishing cycle}  is a 
	simple closed curve $c \subset \Si_f$ that 
	gets contracted to a point when transported to the nodal curve lying over a 
	smooth point of the discriminant $\mbox{Disc}$ of the versal unfolding of 
	$f$. Its class in the homology group $H_1(\Si_f; \ZZ)$ is called {\em 
	algebraic vanishing cycle} or simply a {\em vanishing cycle}.
\end{definition}

\begin{remark}
	It is a consequence of the irreducibility of the discriminant that the set 
	of geometric vanishing cycles forms an orbit by the geometric monodromy 
	group (see \cite{Gab}).
\end{remark}

\subsection*{Geometric vanishing cycles vs. algebraic vanishing cycles}

Here we show to which 
extent the notion of geometric vanishing cycle is finer and much more delicate 
than that of algebraic vanishing cycle. The main tool used here is 
\cite[Theorem B]{Sal} together with some basic facts from mapping class group 
theory. 

\begin{lemma}\label{lem:geom_finer}
	Let $f$ define an isolated plane curve singularity with $g(\Si_f) \geq 2$. 
	Then, for every algebraic vanishing cycle $[c]$ there exists a simple 
	closed curve $c' \in [c]$ that represents the vanishing cycle but such that 
	it is not a geometric vanishing cycle.
\end{lemma}
\begin{proof}
	Let $[c]$ be a vanishing cycle and let $c \in [c]$ be a geometric vanishing 
	cycle representing it which always exists by definition. The hypothesis 
	that 
	$g(\Si) \geq 2$ implies that $f$ is, in particular, not the singularity 
	$A_1$ 
	so $c$ is non-separating. By the Change of coordinates principle, up to an 
	element of $\Mod(\Si)$ we can assume that $c$  is as in 
	\cref{fig:alg_vs_geom} since all 
	non-separating simple closed curves are conjugate. Therefore, there exists 
	a 
	 simple closed curve $c' \in [c]$ such that $c$ and $c'$ bound a genus $1$ 
	 surface. By 
	the 
	homological coherence property \cite[Lemma 2.4]{CalSal}, we find that 
	$\phi_f(c) + \phi_f(c') = 
	2-2g-2 
	=-2$. But by \cite[Theorem B]{Sal} (note that one implication of that 
	theorem holds always without the hypothesis therein stated), $\phi_f(c)=0$ 
	so  $\phi_f(c') \neq 0$ 
	and, by the same result, the curve $c'$ is not a geometric vanishing 
	cycle.
	
	\begin{figure}
		\centering
		\resizebox{0.4\textwidth}{!}{
\begingroup%
  \makeatletter%
  \providecommand\color[2][]{%
    \errmessage{(Inkscape) Color is used for the text in Inkscape, but the package 'color.sty' is not loaded}%
    \renewcommand\color[2][]{}%
  }%
  \providecommand\transparent[1]{%
    \errmessage{(Inkscape) Transparency is used (non-zero) for the text in Inkscape, but the package 'transparent.sty' is not loaded}%
    \renewcommand\transparent[1]{}%
  }%
  \providecommand\rotatebox[2]{#2}%
  \newcommand*\fsize{\dimexpr\f@size pt\relax}%
  \newcommand*\lineheight[1]{\fontsize{\fsize}{#1\fsize}\selectfont}%
  \ifx\svgwidth\undefined%
    \setlength{\unitlength}{218.806198bp}%
    \ifx\svgscale\undefined%
      \relax%
    \else%
      \setlength{\unitlength}{\unitlength * \real{\svgscale}}%
    \fi%
  \else%
    \setlength{\unitlength}{\svgwidth}%
  \fi%
  \global\let\svgwidth\undefined%
  \global\let\svgscale\undefined%
  \makeatother%
  \begin{picture}(1,0.62173333)%
    \lineheight{1}%
    \setlength\tabcolsep{0pt}%
    \put(0,0){\includegraphics[width=\unitlength,page=1]{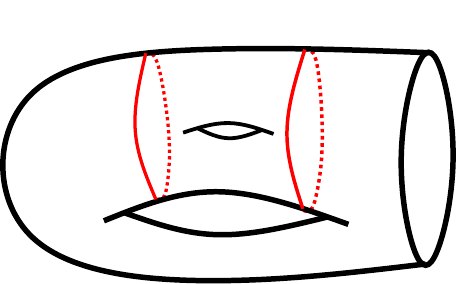}}%
    \put(0.27557254,0.57304118){\color[rgb]{0,0,0}\makebox(0,0)[lt]{\lineheight{1.25}\smash{\begin{tabular}[t]{l}\textbf{$c$}\end{tabular}}}}%
    \put(0.61990067,0.57132598){\color[rgb]{0,0,0}\makebox(0,0)[lt]{\lineheight{1.25}\smash{\begin{tabular}[t]{l}\textbf{$c'$}\end{tabular}}}}%
  \end{picture}%
\endgroup%
}
		\caption{The relative position of the curves $c$ and $c'$. If $c$ is a 
		geometric vanishing cycle, then $c'$ is not; but they represent the 
		same homology 
		class.}
		\label{fig:alg_vs_geom}
	\end{figure}
	
\end{proof}

\subsection*{Vanishing arcs} In this section we define the counterpart to 
vanishing cycles and the central object to this work.

\begin{definition}\label{def:vanishing_arc}
	We say that a class $[a] \in H_1(\Si_f,\partial \Si_f; \ZZ)$ is a 
	{\em vanishing arc} if 
	$V_{f}([\alpha])$ is a vanishing cycle. We say that a single properly 
	embedded arc $a$ is a {\em geometric vanishing arc} if $\Var_{f}(a)$ is a 
	geometric vanishing cycle.
\end{definition}

Since $V_f$ is an isomorphism, the set of vanishing arcs is, by definition the 
preimage by
$V_{f}$ of the set of vanishing cycles.

\subsubsection{Intersection numbers} We briefly recall some properties and 
notation.
Let $a, b \in \calC_{\Si_f} \cup \calI_{\Si_f}$ be two 
closed curves, properly embedded arcs or one of each. We denote by $i(a,b)$  
the geometric intersection number between $a$ and $b$, that is,
\[
i(a,b) = \min_{\substack{a' \sim a  \\b' \sim b}} \# \left(\mathring{a}' \cap 
\mathring{b}'\right)
\]
where $a \sim a'$ is the relation by isotopy in the case of closed curves and 
isotopy relative to the boundary in the case of arcs, and $\mathring{a}$ 
denotes the interior so that only intersection points happening in 
$\mathring{\Si}$ are taken into account.

\begin{remark}
	Observe that the number $i(a,b) $ is always a non-negative integer as it is 
	an 
	unsigned count of intersection points. This classical invariant has been 
	thoroughly used in the literature of mapping class groups and Teichm\"uller 
	spaces. For example the geometric intersection between two curves is 
	crucial on the definition of the curve cumplex in \cite{Harvey}, and in 
	\cite{Harer} the 
	geometric intersection between two arcs was used to define the arc complex. 
	
	For the definition of the geometric intersection between two arcs 
	it is more common to use closed surfaces with marked points or, 
	equivalently, to allow only properly embedded arcs between a finite set of 
	points on each boundary component. We do not require this in this work for 
	the definition of geometric intersection even though a more restricted 
	class of arcs is used later.
\end{remark}

\section{Characterizing vanishing arcs}
\label{sec:characterizing_arcs}
In this section we answer the question of which properly 
embedded arcs in $\calI_\Si$ are sent, by $\Var_f$ to a geometric vanishing 
cycle. This 
characterization is done purely in terms of intersection numbers and depends on 
the extension of the formulas for winding numbers of piecewise $C^1$ curves of 
the previous section. 

\begin{thm}\label{thm:single_arcs}
Let $f$ define an isolated plane curve singularity which is not of type $A_n$ 
or $D_n$ and such that $g(\Si_f) \geq 5$. Let $a \in \calI_{\Si_f}$ be a 
properly embedded arc. Then, $\Var_f(a)$ is a 
geometric vanishing cycle if and only if $i(a,\varphi_f(a))=0.$
\end{thm}

\begin{proof}
By \cite[Theorem B]{Sal}, using the hypothesis on $f$, we have to verify that 
$\Var_f(a)$ is: (i) a simple closed curve, (ii) non-separating, and that (iii) 
$\phi_f(\Var_f(a))=0$.

The hypothesis that  $i(a,\varphi_f(a)) = 0$ implies
that $\Var_f(a)$ is homotopic to a simple closed curve. Using the hypothesis 
that $f$ defines an isolated plane curve singularity, we apply 
\Cref{lem:inverse_separating} to conlcude that $\Var_{f}(a)$ is not a 
separating curve. Since by \cite[Theorem A]{Sal}, $\phi_f$ is in the stabilizer 
of the relative isotopy class of the Hamiltonian vector field $\xi_f$, we get 
that relative winding numbers of arcs are invariant by the geometric monodromy 
and so $\phi_f(a) = \phi_f(\varphi_f(a))$. Then, applying the formula 
\cref{eq:def_C1} from
\Cref{sec:piecewise_c1} to $\Var_f(a)$ which is a concatenation of piecewise 
$C^1$ paths,
\[
\phi_f(\Var_f(a)) = \phi_f(\varphi_f(a)) \pm \pi - \phi_f(a) \mp \pi = 0.
\]
where the signs of $\pm \pi$ and $\mp \pi$ are decided by the discussion of the 
special case in \Cref{sec:piecewise_c1} and they are opposite, that is, $\pm 
\pi \mp \pi = 0.$
This finishes the first part of the proof.

Assume now that $\Var_f(a)$ is a geometric vanishing cycle and so in particular 
it is homotopic to a nonseparating simple closed 
curve.  By the bigon criterion for simple closed curves \cite[Proposition 
1.7]{Farb}, a closed curve 
can be homotoped to a simple  closed 
curve 
if and only if all the self intersections that occur, form bigons. By 
definition, $a$ has no self intersections and so $\varphi_f(a)$ 
has 
no self intersections either. Therefore, there only self intersections of 
$\Var_f(a)$ have to occur between $a$ and $\varphi_f(a)$ and so the only bigons 
that possibly appear are bigons between the two properly embedded arcs. But 
there is also a bigon criterion for properly embedded arcs \cite[Section 
1.2.7]{Farb}. We conclude 
that the homotopy that takes the curve $\Var_f(a)$ to a simple closed curve can 
be made into an homotopy fixing the boundary pointwise.
\end{proof}

\begin{remark}
Let's analyze the different situations when the hypothesis of the above theorem 
are not satisfied.

First, when the singularity is of type $A_n$ or $D_n$, the theorem of Nick 
Salter and the last author (\cite[Theorem A]{Sal}) that is crucially used in 
the proof is not true. However, in this case the authors prove another theorem 
(\cite[Theorem 7.2]{Sal}) 
characterizing geometric vanishing cycles: for $A_n$ singularities the 
geometric vanishing cycles are those
simple closed curves which are invariant (up to isotopy) by the hyperelleptic 
involution; and for $D_n$ singularities these are the simple closed curves that 
are sent to geometric vanishing cycles by the boundary capping map between 
Milnor fibers $\Si(D_n) \to \Si(A_n)$. This criterion gives a sufficient 
condition in this case: for instance, if the hyperelliptic involution $\iota$ 
sends an arc $a$ to $-\varphi(a)$ then $\Var_{f}(a)$ is a geometric vanishing 
cycle.

When the singularity is not of type $A_n$ or $D_n$ and the genus of the Milnor 
fiber is less than $5$, we expect the theorem to be true as stated but, 
technical complications arise in the proof of a result by Aaron Calderon and 
Nick Salter (\cite{CalSal}) used in the proof of \cite[Theorem A]{Sal}. 
Nevertheless, one must notice that these classes of singularities consist of a 
finite and small (only six) collection of topologically different plane curve 
singularities.
\end{remark}

\subsection*{Geometric variation operator on disjoint collections}
As \Cref{ex:separating_wn0} and \Cref{rem:not_bijective} show, this is not the 
end of the story. Next we investigate when the variation 
operator  
takes a disjoint union of properly embedded arcs to a geometric vanishing 
cycle. In this case the theorem is not a full generalization but shows that 
there are no obstructions other than the ones arising from the very properties 
of geometric vanishing cycles.

 Let $I=\{a_1 \ldots, 
a_k\}$ be a collection of disjoint properly embedded piecewise $C^1$ 
arcs $a_i \in \calI_{\Si_f}$ (recall \Cref{rem:good_arc_not_imp}). Then, we 
define 
\[
\Var_{f}(I) = \mathrm{sg}\left( \left\{\Var_{f}(a_1), \ldots, 
\Var_{f}(a_k)\right\}\right).
\]
Where, for a collection of curves $\calC$, the notation $\mathrm{sg}(C)$ 
denotes the collection of simple closed curves that result from applying the 
surgery from 
\cref{fig:surgery} to every intersection (including self-intersections) 
happening in $\bigcup C$. Furthermore, as a consequence of the formula 
\cref{eq:def_C1} and the fact that $\phi_f(-a) = -\phi_f(\varphi_f(a))$ we get, 
denoting $C=\mathrm{sg}\left( \left\{\Var_{f}(a_1), \ldots, 
\Var_{f}(a_k)\right\}\right)$,
\begin{equation}\label{eq:vanishing_wn}
	\sum_{b_i \in \mathrm{sg}(C)}\phi_f(b_i) = 0 .
\end{equation}

\begin{figure}
	\centering
	\resizebox{0.6\textwidth}{!}{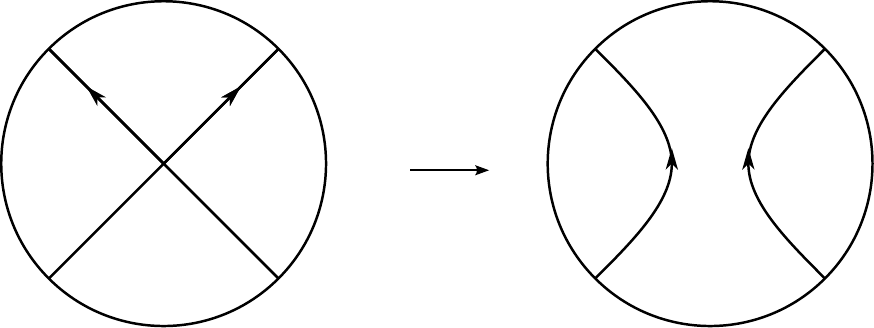}
	\caption{On the left we see a neighborhood around a point of a transverse 
		intersection 
		between two oriented segments belonging to a closed curve in a surface. 
		On the right, we see the neighborhood that 
		substitutes the previous 
		one after surgery is performed.}
	\label{fig:surgery}
\end{figure}

The extension of the definition of $\Var_{f}$ together with 
\cref{eq:vanishing_wn} and the second part of the proof of 
\Cref{thm:single_arcs} prove the following theorem.

\begin{thm}\label{thm:vanishing_arcs_collection}
	Let $I  =  \{a_1, \ldots, a_k\} \subset \calI_{\Si_f}$ be an arcset. Then
	$\Var_f(I)$ is a geometric vanishing cycle if and only if it consists of a 
	single non-separating simple closed curve. 
\end{thm}

\begin{definition} \label{def:geometric_arset}
In the situation of \Cref{thm:vanishing_arcs_collection} above, that is, when 	
$\Var_f(I)$ is a geometric vanishing cycle, we say that $I$ is a geometric 
vanishing arcset.
\end{definition}
\begin{remark}
	Note that the non-separating hypothesis is necessary in this case since  
	\Cref{lem:inverse_separating} only assures that a separating simple closed 
	curve can't be the image of a {\em single} arc. Moreover, the classical 
	result that the variation operator is an isomorphism together with the 
	representation result \Cref{lem:representing_arcs}, suggest that situations 
	like the one described in
	\Cref{ex:separating_wn0} might yield a counterexample. However, we do not 
	have a proof of this at the moment.
\end{remark}

\section{Finding collections of vanishing arcs}

\label{sec:finding_arcs}
In this section we explain how to quickly produce many examples of geometric 
vanishing cycles and arcs for the Brieskorn--Pham singularities $f(x,y) = 
y^p+x^q$ with $\gcd(p,q)=1$. In order to do so, we recall a construction 
already explained in 
\cite{Nor} that gives an 
explicit model for the geometric monodromy of this singularity. In the works 
\cite{Nor,Nor2,Graf}, 
certain ribbon graphs with a metric with a special property (\tat graphs) are 
used but here we 
carry away the construction without entering into those definitions. More 
details can be found in 
the aforementioned papers. 

Let $K_{p,q}$ be the complete bipartite graph of type $p,q$. The 
Milnor fiber $\Si$ of this plane curve singularity retracts to a copy of 
$K_{p,q} \hookrightarrow \Si$. Moreover, if we take two parallel lines on the 
plane, mark $p$ points on one and $q$ points on the other and we join each 
point on one line with all the points of the other, gives an immersion of 
$K_{p,q}$ on the plane in such a way that an immersion on the plane  of 
$\Si$ is given by 
thickening the graph $K_{p,q}$.

Since 
$\gcd(p,q)=1$, then $\partial \Si$ has one boundary component. Hence, $\Si 
\setminus K_{p,q}$ is homeomorphic to $\partial \Si \times (0,1]$. Let 
$\hat{\Si}$ be the compactification $\partial \Si \times [0,1]$. This 
compactification comes with a map
\[
\sigma:\hat{\Si} \to \Si
\]
that is a homeomorphism on $\partial \Si \times (0,1]$ and is genericallly 
$2:1$ 
on $\partial \Si \times \{0\}$. Moreover, $\partial \Si \times \{0\}$ can be 
thought of as a $2pq$-gon where each edge is sent to an edge of $K_{p,q}$ by 
$\sigma$. If each edge of $K_{p,q}$ is identified with a segment of length 
$1/2$, then $\partial \Si \times \{0\} \simeq \RR/ pq \ZZ$, inherits a metric 
and a total length of 
$pq$. Then, by \cite[Examples 2.1 and 3.10]{Nor}, the diffeomorphism
\[
\begin{split}
\hat{\varphi}:\hat{\Si} &\to\hat{\Si} \\
(\theta,t) & \mapsto (\theta + (1-t) ,t)
\end{split}
\]
on the compact cylinder, induces a diffeomorphism on the Milnor 
fiber (in the cited document, every distance is scaled by a factor of $\pi$). 
In other words, the gluing map 
$\sigma|_{\partial \Si \times \{0\}}$ identifies a point $\theta$ 
with a point $\theta'$ if and only if it identifies $\hat{\varphi}(\theta)$ 
with $\hat{\varphi}(\theta')$. Look at \cref{fig:Kpq} to follow this 
discussion. Moreover, this 
induced diffeomorphism 
$\varphi_f:\Si \to \Si$ is 
(in the class of) the geometric monodromy. Therefore, an arc $a$ that 
intersects once 
transversely $K_{p,q}$ at an interior point of an edge, satisfies the 
hypothesis of \Cref{thm:single_arcs} and therefore $\Var_{f}(a)$ is a geometric 
vanishing cycle.

\begin{figure}
	\centering
	\resizebox{0.8\textwidth}{!}{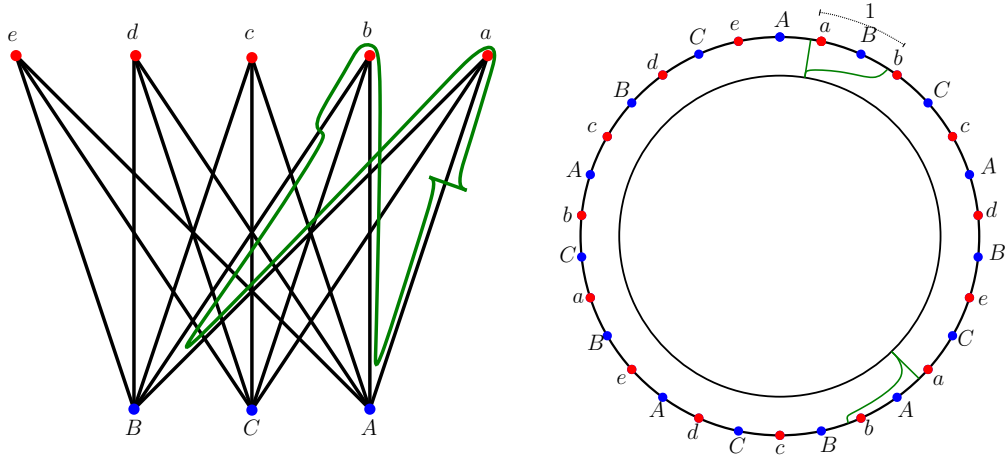}
	\caption{On the left we see the bipartite graph $K_{3,5}$. On the right the 
	surface $\hat{\Si}$ 
	resulting from cutting $\Si$ along the graph $K_{3,5} \hookrightarrow \Si$. 
	On both sides, in 
	green, a properly embedded arc transverse to the graph and its image by the 
	geometric 
	monodromy. Note that \Cref{thm:single_arcs} does not apply here since 
	$g(\Si) = 4$ but bigger bipartite graphs become too cumbersome to draw.}
	\label{fig:Kpq}
\end{figure}

A future project with the collaborator Baldur Sigur{\dh}sson will expose more 
methods to systematically find vanishing arcs as the unstable manifolds of the 
singularities of certain vector field defined on a model of the Milnor fiber.

\section{Topological exceptional collection of vanishing 
arcs}\label{sec:exceptcoll}
  Given an isolated singularity $f$, its generic Morsification $\tilde{f}$ has 
  $\mu$-critical values in $\C$. Fix a regular value $p \in \C$ and choose a
 non-overlapping paths from $p$ to $\mu$-critical values, called vanishing 
 paths. Parallel transports along vanishing paths provide  a distinguished 
 collection of geometric vanishing cycles.
 We may ask if we can find a good collection of arcs (or arcsets) whose 
 geometric variation images form a distinguished collection of geometric 
 vanishing cycles. 
 
   Recall that  a set of disjoint embedded arcs  is called an arcset. 
\begin{definition} \label{definition:linear}
An ordered arcset $\vec{K}=(c_1,\cdots,c_k)$ is called {\em linear} if
\begin{enumerate}
\item $\rho(c_i)$ and $c_i$ only intersect at the boundary points.
\item $\rho(c_{i})$ intersect  transversely $c_{i+1}$ at a single (interior) 
point and does not intersect any other $c_j$'s ($j \neq i, i+1$)
for $1 \leq i \leq \mu -1$ and $\rho(c_{\mu})$ does not intersect any others.
\end{enumerate}
\end{definition}
 The linear condition implies that $\Var_{f}(c_{i})$ is a simple closed curve 
 for any $i$ and they intersect in a successive order. More precisely, 
 $\Var_{f}(c_{i})$ and $\Var_{f}(c_{i+1})$ intersect at one point, which 
 corresponds to the unique intersection point $\rho(c_{i}) \cap c_{i+1}$, for 
 $1 \leq i \leq k-1$. Namely, $\vec{K}$ is linear if its variation image is a 
 linear chain of  $S^1$'s.
 
\begin{lemma}\label{lem:linear}
A linear arcset $\vec{K}=(c_1,\cdots,c_k)$ is a geometric vanishing arcset.
\end{lemma}
\begin{proof}
 
We first make the following observations.
Let $c_{1}$ and $c_{2}$ be simple closed curves intersecting transversely only 
at one point. Then both curves are non-separating because a separating circle
will intersect transversely any other circle even number of times. Also for $C 
= \{ c_{1}, c_{2}\}$, the surgery $\mathrm{sg}(C)$ is again a simple closed 
curve.
$\mathrm{sg}(C)$ is also non-separating: since  $c_{1}$ and $c_{2}$ intersect 
only at one point, a small
translation of $c_1$ (or $c_2$) would intersect $\mathrm{sg}(C)$ at one point 
only as well.
 
Moreover, $\mathrm{sg}(C_{1} \cup C_{2}) = \mathrm{sg}(\mathrm{sg}(C_{1}), 
\mathrm{sg}(C_{2}))$  for any collections $C_{1}$ and $C_{2}$.
Therefore, we can sucessively apply the above argument to prove that 
$\mathrm{sg}(\vec{K})$ is a simple closed curve which is non-separating. So we 
conclude by \Cref{thm:vanishing_arcs_collection}.
\end{proof}

We remark that the converse does not hold in general.
In our main applications, we will find linear arcsets, which are   geometric 
vanishing arcsets by the above lemma.

Now, let us introduce a notion from \cite{BCCJ23} that is analogous to a 
distinguished collection of vanishing cycles.
\begin{definition}[Topological exceptional collection] 
\label{definition:exceptional}
An ordered collection of geometric vanishing arcsets 
$(\vec{K}_1,\ldots,\vec{K}_\mu)$ is called a {\em topological exceptional 
collection} if the following holds.
\begin{enumerate}
\item Any two arcsets in the collection are disjoint from each other.
\item For any $i <j$, we have 
$$\varphi_f(\vec{K}_j) \bullet \vec{K}_i =0.$$
\item For any $i$,
$$\vec{K}_i \bullet \Var_f(\vec{K}_i) =-1.$$
\end{enumerate}
\end{definition}
\begin{remark}
In the above formulas, we  take the sum of all intersection numbers of involved 
arcs.
Intersection numbers between different arcsets are well-defined since their 
endpoints  are disjoint from each other.
\end{remark}

One justification of the above definition is the following.
\begin{cor}
The Seifert form is non-degenerate and triangular with respect to the basis 
$$\{\Var_f(\vec{K}_1),\dots, \Var_f(\vec{K}_\mu)\}.$$
\end{cor}
\begin{proof}
Seifert form can be defined by the intersection number between the preimage of 
a variation operator and the vanishing cycle. 
By definition, we have $\Var_f(\vec{K}_j) \bullet \vec{K}_i =0$ for $i<j$ 
because $\vec{K}_j \cap \vec{K}_i = \emptyset$.
This implies the claim.
\end{proof}
For convenience, we will not write the arrow in $\vec{K}$ and just write $K$ 
from now on. 

Here is one warning for a possible confusion. The variation image of a  
vanishing arc is a vanishing cycle of a certain Morsification of $f$.
When we have a topological exceptional collection, our definition does {\em 
not} imply that  the corresponding collection of vanishing cycles come from a 
single Morsification  of $f$.
But we conjecture that they do.
\begin{conjecture}
The ordered collection of vanishing cycles from a topological exceptional 
collection of geometric vanishing arcsets is isotopic to the distinguished 
collection of vanishing cycles of a Morsification of $f$.
\end{conjecture}

The next theorem  provides one example supporting the conjecture.

%

\begin{thm}\label{thm:maina}
Given any A'Campo divide of a plane curve singularity,
we have a topological exceptional collection of geometric vanishing arcsets.
Furthermore, their variation images form a distinguished collection of 
vanishing cycles that are described by A'Campo \cite{Areal} (up to isotopy).
\end{thm}
In the rest of the paper, we prove the above theorem.
%

\section{Prelimiaries on A'Campo divide}\label{sec:divide}
In this section, we recall the notion of A'Campo divide ( \cite{ACampo1975}, 
\cite{Areal}).
\subsection*{A'Campo divide}
Let $f : \mathbb{C}^{2} \to \mathbb{C}$ be a totally real isolated plane curve 
singularity. i.e. $f$ can be written as a product of irreducible real factors 
(as complex functions). Then there exists a deformation of each factor so that 
their product gives a real Morsification $\{ f_{t} \}_{0 \leq t \leq t_{0}}$ of 
$f$ where $t_{0}$ is a sufficiently small positive real number. The existence 
of real Morsifications for totally real plane curve singularities is a 
classical result independently proven by Norbert A'Campo \cite{ACampo1975} and 
by \cite{GZ74}. It is remarkable that it is still a conjecture the existence of 
real Morsifications for real plane curves, that is, when $f$ is a real 
polynomial that does not factor into its irreducible components as a complex 
polynomial over the real (see \cite{Lev18} for some advances made in this 
direction).

\begin{definition}
Fix one real Morsification $\{ f_{t} \}_{0 \leq t \leq t_{0}}$ of $f$. Then, an 
{\em A'Campo divide} $\mathbb{D}_{f}$ of $f$ is defined to be
$$\mathbb{D}_{f} := f_{t_{0}}^{-1}(0) \cap B_{\epsilon}(0) \hookrightarrow 
B_{\epsilon}(0)$$
for some positive real number $\epsilon$ satisfying the following equation:
$$\mu = 2d - r + 1$$
where $\mu$ is the Milnor number of $f$, $d$ is the number of double points in 
the interior of $B_{\epsilon}(0) \subset \RR^2$, and $r$ is the number of 
irreducible factors of $f$.
\end{definition}
We remark that for a given $f$, a different choice of  $\{ f_{t} \}$  may 
result in a different A'Campo divide. We will choose and fix it.

\begin{definition}[$A\Gamma$ diagram \cite{GZ1974},  \cite{ACampo1975}] 
\label{definition:AG}
Let $\mathbb{D}_{f}$ be an A'Campo divide of $f$. The complement 
$B_{\epsilon}(0) \setminus f_{t_{0}}^{-1}(0)$ consists of finitely many 
components. Among them, a component whose boundary does not intersect $\partial 
B_{\epsilon}(0)$ is called a bounded region. Then, we call any bounded region a 
$+$ or $-$ region according to the value of $f_{t_{0}}$ on that component. The 
A'Campo--Gusein-Zade diagram ($A\Gamma$ diagram for short) 
$A\Gamma(\mathbb{D}_{f})$ of $\mathbb{D}_{f}$ is defined to be the planar graph 
as follows.
\begin{enumerate}
\item A set of vertices: three kinds of vertices.
\begin{itemize}
\item Each double point in $\mathbb{D}_{f}$ gives a $0$ vertex.
\item Each bounded region of $\mathbb{D}_{f}$ gives a $+$ or $-$ vertex 
depending on the value of $f_{t_{0}}$.
\item No vertex for any unbounded region (connected component which is not 
bounded).
\end{itemize}
\item A set of edges:
\begin{itemize}
\item There is an edge between a $+$ vertex and a $-$ vertex if the 
intersection between two closures of corresponding $+$ region and $-$ region is 
a line segment (connecting two double points). 
\item There is an edge between a 0 vertex and a $+$ (or $-$) vertex if the 
corresponding double point in $\mathbb{D}_{f}$ contained in the closure of 
corresponding $+$ (or $-$) region. 
\item No edge between vertices of the same type.
\end{itemize}
\end{enumerate}
For a vertex $v$ of $A\Gamma(\mathbb{D}_{f})$, we denote its type (+, 0, or -) 
by $|v|$.
\end{definition}

Let $n_{-}, n_{0}$, and $n_{+} \geq 0$ be the number of $-$, $0$, and $+$ 
vertices in $A\Gamma(\mathbb{D}_{f})$, respectively. 
We can choose the ordering on the set of these vertices such that
\begin{itemize}
\item the order is chosen arbitrarily among the vertices of the same type and
\item $- \mbox{ vertices} < 0 \mbox{ vertices} < + \mbox{ vertices}$ between 
vertices of different types.
\end{itemize}
Then, we have the following ordered set of vertices:
\begin{equation}\label{eq:vertices}
\{v^{-}_{1}, \dots, v^{-}_{n_{-}}, v^{0}_{1} ,\dots, v^{0}_{n_{0}}, v^{+}_{1}, 
\dots, v^{+}_{n_{+}} \}.
\end{equation}

\begin{figure}
	\centering \includegraphics[scale=0.5]{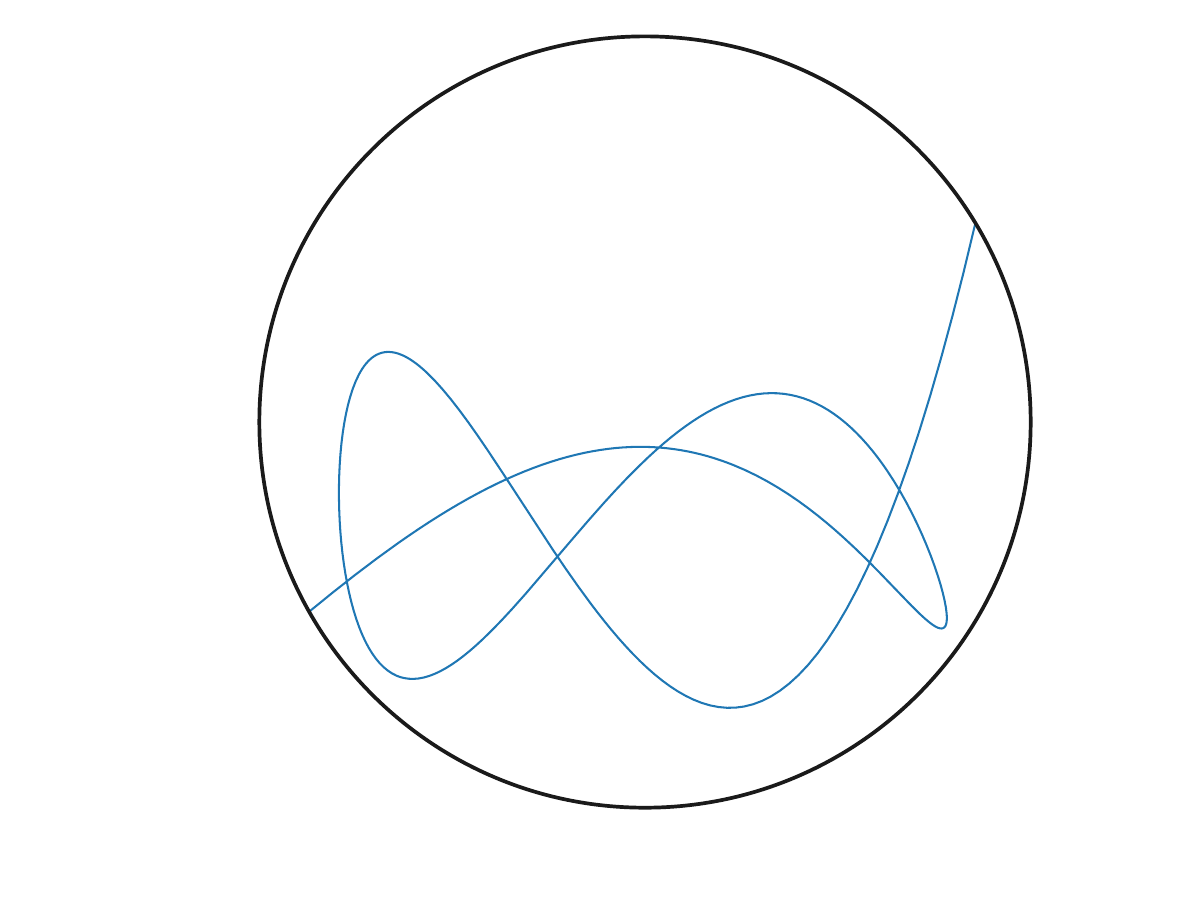}
	\caption{An example of a divide for the plane curve defined by $-x^{8} - 
	x^{7} - 3 x^{5} y + y^{3}$ using Gusein-Zade method via Chebyshev 
	polynomials.}
	\label{fig:divide_478}
\end{figure}

\begin{thm}[\cite{Areal}]\label{thm:ACampo}
Given a divide $\mathbb{D}_{f}$ of $f$, there exists a set of vanishing paths 
such that 
the corresponding distinguished collection of vanishing cycles 
$$\overrightarrow{V}_{f} = (V^{-}_{1}, \dots, V^{-}_{n_{-}}, V^{0}_{1}, \dots, 
V^{0}_{n_{0}}, \dots, V^{+}_{1}, \dots, V^{+}_{n_{+}})$$ in the Milnor fiber of 
$f$ satisfies the following properties:
\begin{enumerate}
\item The ordered set of vertices \eqref{eq:vertices} of $A\Gamma 
(\mathbb{D}_{f})$ corresponds to the 
distinguished collection $\overrightarrow{V}_{f}$ of vanishing cycles for this 
set of vanishing paths.
\item Two vanishing cycles intersect exactly at one point if and only if there 
is an edge connecting the corresponding two vertices in $A\Gamma 
(\mathbb{D}_{f})$.
Moreover, one can orient vanishing cycles so that we get
$$V^+_{i} \bullet V^0_{j} =  V^0_{j} \bullet V^-_{k}= V^+_{k} \bullet V^-_{i}= 
+1$$
for any $1 \leq i \leq n_{+}$, $1 \leq j \leq n_{0}$, and $1 \leq k \leq n_{-}$ 
whenever they intersect.
\end{enumerate}
When we regard the order only, we omit types and denote the vanishing cycles by 
just $(V_{1}, \dots, V_{\mu})$.
\end{thm}


Then, the geometric monodromy $\varphi_{f} : \Sigma_{f} \to \Sigma_{f}$ can be 
represented by 
$$\varphi = \tau_{V_{1}} \circ \dots \circ \tau_{V_{\mu}}$$
where $\tau_{V_{i}}$ is the right Dehn twist along the vanishing cycle $V_{i}$.

 \subsection*{Depth of a divide}
We recall the notion of depth from \cite{BCCJ23}.
%
%
\begin{definition} \label{definition:depth}
The {\em depth} of vertices in $A\Gamma(\mathbb{D}_{f})$ is defined as follows.
\begin{enumerate}
\item  If a vertex $v$ is contained in the closure of some unbounded region in 
$B_{\epsilon}(0) \setminus f_{t_{0}}^{-1}(0)$, then $v$ has depth 0.
\item Remove all depth 0 vertices and all edges connected to them from 
$A\Gamma(\mathbb{D}_{f})$. We get a new diagram $A\Gamma_{1}(\mathbb{D}_{f})$. 
Then, a vertex $v$ of $A\Gamma(\mathbb{D}_{f})$ has depth 1 if
$v$ is contained in $A\Gamma_{1}(\mathbb{D}_{f})$ and it is a depth $0$ vertex 
of $A\Gamma_{1}(\mathbb{D}_{f})$.
\item Inductively, delete all vertices of depth less than $k$ and adjacent 
edges from $A\Gamma(\mathbb{D}_{f})$. Then, we obtain a new diagram 
$A\Gamma_{k}(\mathbb{D}_{f})$. A vertex $v$ of $A\Gamma(\mathbb{D}_{f})$ has 
depth $k$ if
$v$ is contained in $A\Gamma_{k}(\mathbb{D}_{f})$ and it is a depth $0$ vertex 
of $A\Gamma_{k}(\mathbb{D}_{f})$.
\end{enumerate}
We denote the depth of a vertex $v$ by $\mathrm{dep}\,v$. The depth of the 
$A\Gamma(\mathbb{D}_{f})$ (and $\mathbb{D}_{f}$) is defined to be the maximum 
of the set $\{ \mathrm{dep}\,v \, |\, v \in A\Gamma(\mathbb{D}_{f}) \}$.
\end{definition}
The depth of vanishing cycle $V^{\bullet}_{i}$  is defined as that of the 
corresponding vertex $v^{\bullet}_{i}$ in $A\Gamma(\mathbb{D}_{f})$.

\begin{figure}[h]
\includegraphics[scale=1]{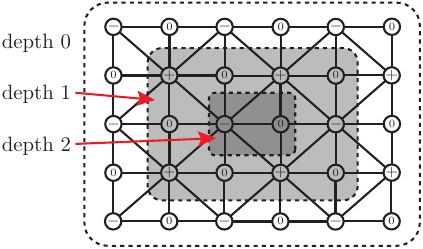}
\centering
\caption{Example of $A\Gamma$ diagram and depth.}
\label{fig:depthex}
\end{figure}

\subsection*{Milnor fiber and vanishing cycles from an A'campo divide}
 
A'Campo \cite{Areal} gave a  combinatorial model of the Milnor fiber $M_{f}$ 
from a divide (or from $A\Gamma(\mathbb{D}_{f})$ diagram).

For each double point   $v$  in $\mathbb{D}_{f}$ (which is each $0$ vertex of 
$A\Gamma(\mathbb{D}_{f})$), 
consider a surface $F_v$ given as in \cref{fig:block}. 
\begin{figure}[h]
\includegraphics[scale=0.7]{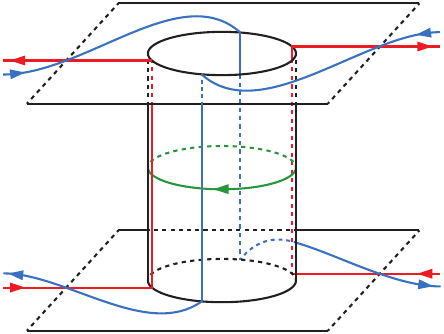}
\centering
\caption{The building block $F_v$.}
\label{fig:block}
\end{figure}
(Imagine two ribbons in skew position and take the connect sum between them. 
Upper and lower ribbons are placed along $f^{-1}(0)_{t_0}$ and they cross each 
other at $v$.
Regarding the $+$ region as the first quadrant of the plane, $x$-axis is the 
upper ribbon and $y$-axis is the lower ribbon. Given the same orientation for 
each $F_v$.
see \cref{fig:arcex} for example.)

For  two double points $v,w$ that are joined an edge in $f^{-1}(0)_{t_0}$,  we 
glue  $F_v$ and $F_w$ along the corresponding dotted boundaries
(with a half twist to match the orientations).

A'Campo showed that the resulting surface is diffeomorphic to the Milnor fiber 
$\Sigma_f$. Solid boundaries of each $F_v$ are glued to form the boundary of 
the Milnor fiber $\partial \Sigma_f$.

An important feature of this model is that distinguished collection of 
vanishing cycles are built in.
Namely, vanishing cycle corresponding to the 0 vertex $v$ is the 1-cycle in the 
middle cylinder of $F_v$, drawn as a green circle in \cref{fig:block}.
For each $+$ bounded region of the divide, we have the corresponding vanishing 
cycle which winds around the $+$ region. This is given by  the red arcs that 
are glued along the $+$-region.
Similarly, for each $-$ bounded region of the divide,  corresponding vanishing 
cycles are locally drawn as the blue arcs in \cref{fig:block}.

%
When a region is unbounded, there is no associated vanishing cycle. To indicate 
this, we will omit the corresponding red/blue arc from the picture.
(see \cref{fig:K-} where red arc on the right is omitted).  

\begin{figure}[h]
\includegraphics[scale=0.7]{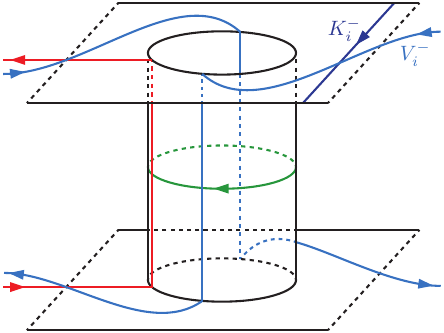}
\centering
\caption{The building block $F_v$.}
\label{fig:K-}
\end{figure}

\section{Proof of \Cref{thm:maina} for depth 0 cases}\label{sec:depthzero}
 
\Cref{thm:maina} for the  depth 0 cases was essentially proved in  
\cite{BCCJ23} and we will recall the construction therein.
In this case, each geometric vanishing arcset  consists of a single properly 
embedded arc.
Later in the general case, a geometric vanishing arcset for a vertex of depth 
$d$ will consist of $d+1$ disjoint properly embedded arcs.


\subsection*{Adapted family of arcsets}
We first recall the notion of an adapted family, which is quite convenient for 
(topological) variation operator calculations.
Namely, for a distinguished collection of vanishing cycles $(V_1,\dots, 
V_\mu)$, one would like to find a collection of  vanishing arcsets $(K_1,\dots, 
K_\mu)$ satisfying  $$V_{f}([K_{i}]) = [V_{i}] \in H_{n-1}(M) \; \textrm{for 
all} \; 1 \leq i \leq \mu.$$

 \begin{definition} \label{definition:adapted}\cite{BCCJ23}
A collection of arcsets $(K_1,\dots, K_\mu)$  is called {\em adapted} to the 
distinguished collection of vanishing cycles $(V_1,\cdots, V_\mu)$ if
it satisfies the following intersection conditions.
\begin{enumerate}
\item For any $j > i$, $K_j \bullet V_i =  -( -1)^\frac{n(n+1)}{2} V_{j} 
\bullet V_{i}$.
\item For any $j < i$, $K_j \bullet V_i = 0$.
\item For any $j$, $K_j \bullet V_j = 1$.
\end{enumerate}
The sign in (i) is due to the well-known Picard--Lefschetz formula (we follow 
the convention of \cite{AGZVII} with $f:\mathbb{C}^n \to \C$ and $n=2$ in our 
case).
\end{definition}
The following Proposition was shown in \cite[Proposition 6.7]{BCCJ23}.
\begin{prop}
\label{prop:adaptedvar}
If $(K_1,\dots, K_\mu)$ is adapted to $(V_1,\dots, V_\mu)$, then  we have
$$V_{f} ([K_i]) = ( -1)^\frac{n(n+1)}{2} [V_i], \;\;\; \forall i=1, \dots, 
\mu.$$ 
\end{prop}
Since $n=2$ in our case, we take the orientation reversal  $\overline{K}_i$ to 
obtain
$$V_f[\overline{K}_i] =  [V_i].$$
Adapted collection is related to the exceptional collection as follows.
\begin{prop} \label{prop:excep}
Suppose a collection $(K_1,\dots, K_\mu)$ of  arcsets satisfy the following:
\begin{enumerate}
\item Arcs in $K_1,\dots, K_\mu$ are disjoint from each other.
\item $ (K_1,\dots, K_\mu)$ is adapted to a distinguished collection of 
vanishing cycles $ (V_1,\cdots, V_\mu)$. 
\item $K_i$ is a geometric vanishing arcset for each $i$.
\end{enumerate}
Then,  $(\overline{K}_1,\dots, \overline{K}_\mu)$ is a topological exceptional 
collection of geometric vanishing arcsets.
\end{prop}
\begin{proof}
The disjoint condition (i) in \Cref{definition:exceptional} holds by the 
assumption. Since 
$$[V_i] = V_{f}([\overline{K}_i]) = [\varphi_f(\overline{K}_j)] -  
[\overline{K}_j] $$ and all $K_i$'s are disjoint, the vanishing conditions 
(ii)  \Cref{definition:adapted} is equivalent to the condition (ii)  in 
\Cref{definition:exceptional}. Lastly we have $\overline{K}_i \bullet V_i = - 
K_i \bullet V_i = -1$.
\end{proof}
 

\subsection*{Proof of  depth 0 case} \label{subsec:depth0}
In the case of depth 0 vertices, we will find a collection of properly embedded 
arcs $(K_1,\dots,K_\mu)$ which satisfies the assumptions of  
\Cref{prop:adaptedvar}.
Namely, each $K_i$ consists of a single arc, such that $\varphi_f(K_i)$ and 
$K_i$ do not intersect in the interior of $\Sigma_f$ (i.e. $i(K_i, 
\varphi_f(K_i))=0$). By  \Cref{thm:single_arcs}
each $K_i$ is a geometric vanishing arc. The adapted condition will help us to 
choose the corresponding arc and guarantees that topological variation operator 
takes $\overline{K}_i$ to $V_i$. Since the global monodromy is the composition 
of Dehn twists, it is not difficult to check directly that geometric variation 
image of $\overline{K_i}$ is not only homologous but also isotopic to the 
vanishing cycle $V_i$ of A'Campo.

\begin{itemize}
\item $-$ vertex

Let $V^{-}_{i}$ be a vanishing cycle which corresponds to a depth $0$ vertex of 
type $-$. Then, we need to find a curve $K^{-}_{i}$ satisfying 
$$K^{-}_{i} \bullet V^{\bullet}_{j} = \begin{cases} 1 & (\bullet = -, \; i=j), 
\\ 0 &(\mbox{otherwise}). \end{cases}$$
Since it is of depth $0$, the corresponding negative region is neighboring a 
positive unbounded region (hence a missing red arc) and these two regions share 
a $0$-vertex.
In the building block for this $0$-vertex, we have illustrated a part of the 
corresponding vanishing cycle $V^{-}_{i}$ in \cref{fig:K-}. 
Note that the missing red arc allows a room to draw the curve $K_i^-$ with the 
right intersection condition.

\item $0$ vertex

For a vanishing cycle $V^{0}_{i}$, depth 0 implies that it is neighboring an 
unbounded $-$ or $+$ region.  In these two cases, $K^{0}_{i}$ should satisfy
$$K^{0}_{i} \bullet V^{\bullet}_{j} = \begin{cases} 1 & (\bullet = 0, \; i=j), 
\\  1 & (\bullet=-, V^{-}_{j}\mbox{ and }V^{0}_{i}\mbox{ are connected in 
}A\Gamma \mbox{ diagram}),  \\ 0 &(\mbox{otherwise}) \end{cases}$$
and it is given as in \cref{fig:K0} (without $V^-$ on the left and without 
$V^+$ on the right) 

\begin{figure}[h]
\begin{subfigure}[t]{0.43\textwidth}
\includegraphics[scale=0.7]{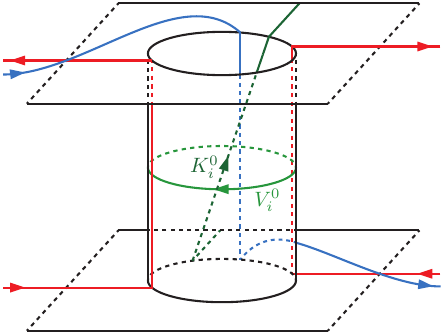}
\centering
\caption{ }
\label{fig:K01}
\end{subfigure}
\begin{subfigure}[t]{0.43\textwidth}
\includegraphics[scale=0.7]{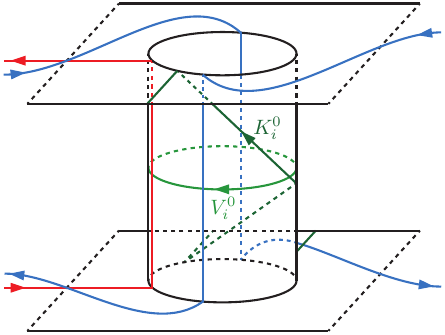}
\centering
\caption{ }
\label{fig:K02}
\end{subfigure}
\centering
\caption{$K^{0}_{i}$ in the building block.}
\label{fig:K0}
\end{figure}

\item $+$ vertex

A non-compact Lagrangian $K^{+}_{i}$ for a vanishing cycle $V^{0}_{i}$ satisfies
$$K^{+}_{i} \bullet V^{\bullet}_{j} = \begin{cases} 1 & (\bullet = +, \; i=j), 
\\ 
1 & (\bullet=0, V^{0}_{j}\mbox{ and }V^{+}_{i}\mbox{ are connected in }A\Gamma 
\mbox{ diagram}), \\
1 & (\bullet=-, V^{-}_{j}\mbox{ and }V^{+}_{i}\mbox{ are connected in }A\Gamma 
\mbox{ diagram}),  \\
0 &(\mbox{otherwise}). \end{cases}$$
We can find a part where $V^{+}_{i}$ lives alone as in the $-$ case. Then, we 
cut $V_i^+$ near that part and attach two new ends to the boundary of the 
Milnor fiber as drawn in \cref{fig:K+}.

\begin{figure}[h]
\includegraphics[scale=0.7]{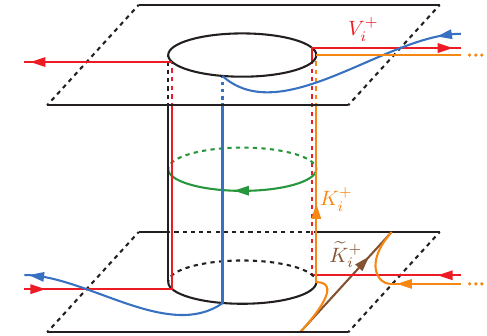}
\centering
\caption{$K^{+}_{i}$ in the building block.}
\label{fig:K+}
\end{figure}
\end{itemize}

One can check that these collection of properly embedded arcs $K_1,\cdots, 
K_\mu$ are all disjoint, adapted and $i(K_i, \varphi_f(K_i))=0$ for any $i$.
This proves  \Cref{thm:maina} for depth 0 cases.

\section{Basic arcs for higher depth cases}\label{sec:basic}
We will prove  \Cref{thm:maina} for a divide of non-zero depth in the remaining 
sections.
For a vertex $v_i$ of depth $d$ ($d>0$), there exist a path of $d$-edges in 
$A\Gamma(\mathbb{D}_{f})$-diagram connecting $v_i$ to an outer vertex (of depth 
$0$).
Geometric vanishing arcset $K_i$ for a vertex $v_i$ will be given by the set of 
basic arcs associated to the edges in this path.
In this section, we will define these basic arcs.  In the next section, we will 
prescribe how to choose paths for vertices in $A\Gamma(\mathbb{D}_{f})$-diagram
so that the associated geometric vanishing arcsets form a topological 
exceptional collection.

\subsection*{Relevant vertices}
We will first define the notion of relevant vertices of a given edge for 
convenience (relevant for intersection calculations later on).
Recall that there are three kinds of edges in $A\Gamma(\mathbb{D}_{f})$: 
between $+$ and $0$ vertices, between $+$ and $-$ vertices, and between $0$ and 
$-$ vertices. 
\begin{definition} \label{definition:relevant}
For two vertices $v$ and $w$ of $A\Gamma(\mathbb{D}_{f})$, we say that $v$ is 
adjacent to $w$ (and vice versa) if $v = w$ or there exists an edge connecting 
$v$ and $w$.
Then, we choose a subset of vertices of $A\Gamma(\mathbb{D}_{f})$ called {\em 
relevant vertices} associated to an edge in $A\Gamma(\mathbb{D}_{f})$ as 
follows.
\begin{enumerate}
\item For an edge $e$  connecting $v_{+}$ and $v_{0}$,
$$\mathcal{R}_{e} := \left \{ v \mid v \mbox{ is adjacent to }v_{+}\mbox{ but 
not adjacent to }v_{0} \right \} \cup \left \{ v_{+} \right \} .$$
\item For an edge $e$  connecting $v_{+}$ and $v_{-}$,
$$\mathcal{R}_{e} := \left \{ v \mid v \mbox{ is adjacent to }v_{+} \right \} 
\setminus \left \{ v_{-} \right \}.$$
\item For an edge $e$  connecting $v_{0}$ and $v_{-}$,
$$\mathcal{R}_{e} := \left \{ v \mid v \mbox{ is adjacent to }v_{0}\mbox{ but 
not adjacent to }v_{-} \right \} \cup \left \{ v_{0} \right \}.$$
Note that if any + vertex is adjacent to $v_{0}$, it is also adjacent to 
$v_{-}$. Therefore, that $+$ vertex is not in $\mathcal{R}_{e}$.
\end{enumerate}
In addition, we also define relavant vertices for $v_+$ or $v_-$ of depth 0.
\begin{enumerate}
\item[(iv)] For a vertex $v_{+}$ of depth 0,  
$$\mathcal{R}_{v} := \left \{ v \mid v \mbox{ is adjacent to }v_{+} \right \}.$$
\item[(v)] For a vertex $v_-$ of depth $0$,
$$\mathcal{R}_{v} := \left \{ v_{-} \right \}.$$
\end{enumerate}
In each case, the rest of the vertices not in $\mathcal{R}_{v}$ are called 
irrelevant. The vanishing cycles that correspond to the relevant (or 
irrelevant) vertices of given edge or vertex are called the {\em relevant (or 
irrelevant) vanishing cycles} of the edge or vertex and we say that the 
vanishing cycle is relevant (or irrelevant) to given edge or vertex.
\end{definition}

\begin{figure}[h]
\includegraphics[scale=1]{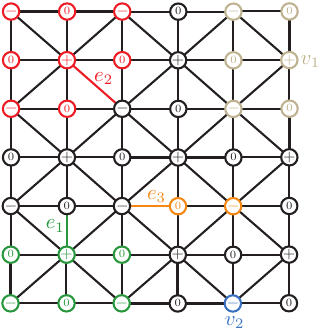}
\centering
\caption{Examples of relevant vertices.}
\label{fig:relevant}
\end{figure}
See \cref{fig:relevant} for an example of the 5 cases in the definition (in the 
order of $e_1,e_2,e_3, v_1,v_2$).
\subsection*{Basic arcs}
We now define the associated basic arc for each edge of  
$A\Gamma(\mathbb{D}_{f})$.
\begin{lemma} \label{lem:+0-}
For any edge $e$ in $A\Gamma(\mathbb{D}_{f})$, there exists an arc $K$ in $M$ 
such that
$$K \bullet V = \begin{cases}
-1 & (V \in \mathcal{R}_{e}), \\
0 & (V \notin \mathcal{R}_{e}).
\end{cases}$$
These arcs are denoted by $K^{+,0}$,$K^{+,-}$, or $K^{0,-}$ according to the 
type of an edge.
We denote by $K^{0,+}$, $K^{-,+}$, and $K^{-,0}$ the orientation reversals of 
the above respectively.
\end{lemma}
\begin{proof}
 Let us start with the case $K^{+,0}$. Let $e$ be an edge between $v_+$ and 
 $v_0$.
We define $K^{+,0}$ using the vanishing cycle $V^+$ as in  \Cref{subsec:depth0}.
Namely, consider the building block $F_{v_0}$. We define $K^{+,0}$ as an arc 
which starts and ends as drawn in  \cref{fig:K+0K+-} (a) and travels along the 
vanishing cycle $V^+$ in the Milnor fiber. More precisely, $K^{+,0}$ lies in a 
small neighborhood $V^+$ in $\Sigma_f \setminus F_{v_0}$, with only one 
negative intersection $K^{+,0} \bullet V^+=-1$ therein. (Note that there should 
be at least one crossing because  starting and ending segments lie on different 
sides.)
As a result, $K^{+,0}$ do not intersect two $V^-$'s and $V^0$ that appear in 
$F_{v_0}$, and the rest of the vanishing cycle $V_j$'s,  $K^{+,0} \bullet V_j = 
- V_+ \bullet V_j$.
This proves that $K^{+,0}$ satisfies the desired intersection condition in the 
lemma.
 
\begin{figure}[h]
\begin{subfigure}[t]{0.48\textwidth}
\includegraphics[scale=0.7]{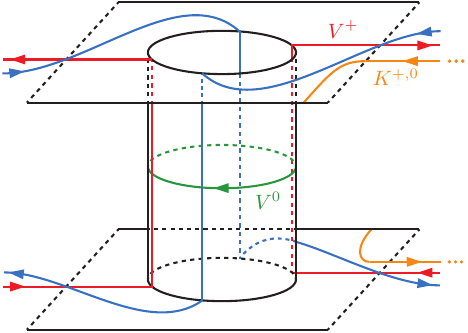}
\centering
\caption{$K^{+,0}$ and corresponding $V^{+}$, $V^{0}$.}
\label{fig:K+0}
\end{subfigure}
\begin{subfigure}[t]{0.48\textwidth}
\includegraphics[scale=0.7]{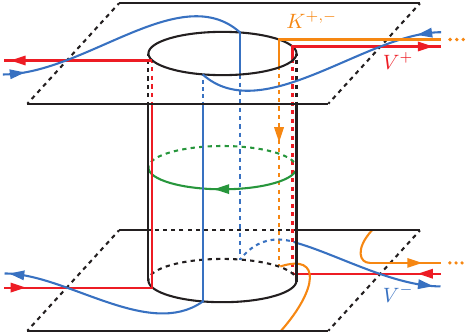}
\centering
\caption{$K^{+,-}$ and corresponding $V^{+}$, $V^{-}$.}
\label{fig:K+-}
\end{subfigure}
\centering
\caption{Description of basic arcs $K^{+,0}$ and $K^{+,-}$.}
\label{fig:K+0K+-}
\end{figure}
The other cases can be handled similarly.  For an edge connecting $v^+$ and 
$v^-$, $K^{+,-}$ is defined similarly using \cref{fig:K+0K+-} (b) and $V^+$.
Note that $K^{+,-}$ intersects all vanishing cycles intersecting $V^{+}$ (and 
$V^{+}$ itself) except $V^{-}$ of the edge.

For an edge connecting $v^0$ and $v^-$, $K^{0,-}$ drawn in \cref{fig:K0-} only 
intersects $V^{0}$ and negative vanishing cycle which is not $V^-$. These are 
exactly the relevant vanishing cycles of the given edge. This proves the lemma.
\begin{figure}[h]
\includegraphics[scale=0.7]{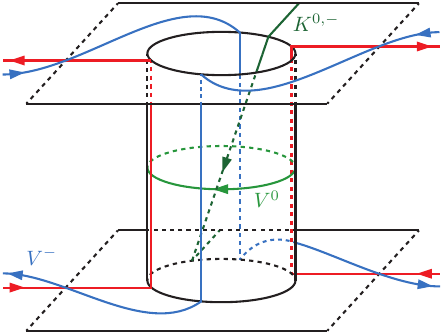}
\centering
\caption{$K^{0,-}$ and corresponding $V^{0}$, $V^{-}$.}
\label{fig:K0-}
\end{figure}
\end{proof}
We call the 6 type of arcs in the above lemma as {\em basic arcs}.

We may rephrase our choice of  arcs for depth $0$ vertices in 
\Cref{subsec:depth0} as follows.
\begin{lemma} \label{lem:+-}
Let $v$ be a depth 0 vertex of type $+$ or $-$ in $A\Gamma(\mathbb{D}_{f})$. 
Then, there exists a properly embedded arc $K$ in $M$ such that
$$K \bullet V = \begin{cases}
1 & (V \in \mathcal{R}_{v}), \\
0 & (V \notin \mathcal{R}_{v}).
\end{cases}$$
\end{lemma}

\subsection*{Monodromy images of basic arcs}
Now, we describe the monodromy images of basic arcs, which are needed later. 
Recall that the monodromy $\varphi_{f}$ is the composition of Dehn twists along 
vanishing cycles:
$$\varphi_{f} = \tau_{V^{-}_{1}} \circ \dots \circ \tau_{V^{+}_{n_{+}}}.$$

A priori, the basic arc $K^{+,0}$ meets $V^{+}$ and does not meet any other $+$ 
vanishing cycles. After taking the Dehn twist $\tau_{V^{+}}$, 
$\tau_{V^{+}}(K^{+,0})$ becomes an arc given in \cref{fig:K+0mon} (after some 
isotopy). Then, the arc $\tau_{V^{+}}(K^{+,0})$ intersects only $V^{0}$ among 
$0$ vanishing cycles. Its Dehn twist image $\tau_{V^{0}} \circ 
\tau_{V^{+}}(K^{+,0})$ is drawn in \cref{fig:K+0mon} and it does not meet any 
$-$ vanishing cycle. Thus, the monodromy image $\varphi_{f}(K^{+,0})$ is the 
same as $\tau_{V^{0}} \circ \tau_{V^{+}}(K^{+,0})$ as a result.

\begin{figure}[h]
\begin{subfigure}[t]{0.48\textwidth}
\includegraphics[scale=0.7]{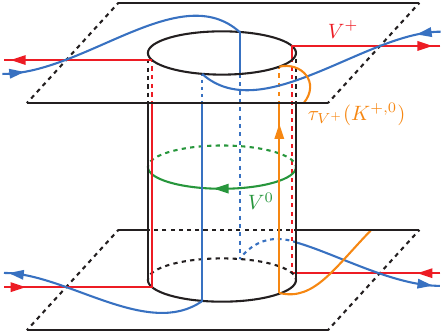}
\centering
\caption{$\tau_{V^{+}}(K^{+,0})$.}
\end{subfigure}
\begin{subfigure}[t]{0.48\textwidth}
\includegraphics[scale=0.7]{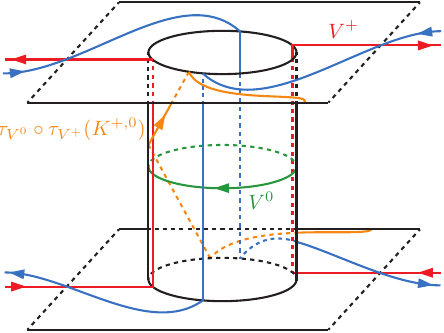}
\centering
\caption{$\tau_{V^{0}} \circ \tau_{V^{+}}(K^{+,0})$.}
\end{subfigure}
\centering
\caption{Monodromy image of $K^{+,0}$.}
\label{fig:K+0mon}
\end{figure}

Similarly, the basic arc $K^{+,-}$ meets $V^{+}$ and does not meet any other 
$+$ vanishing cycles and its Dehn twist image $\tau_{V^{+}}(K^{+,-})$ is a 
straight line connecting two endpoints of $K^{+,-}$ (see \cref{fig:K+-mon}).  
This now meets $V^{-}$ and does not meet any other $-$ vanishing cycles and the 
monodromy image $\varphi_{f}(K^{+,-})$ is given (after some isotopy) as in 
\cref{fig:K+-mon}.

\begin{figure}[h]
\begin{subfigure}[t]{0.48\textwidth}
\includegraphics[scale=0.7]{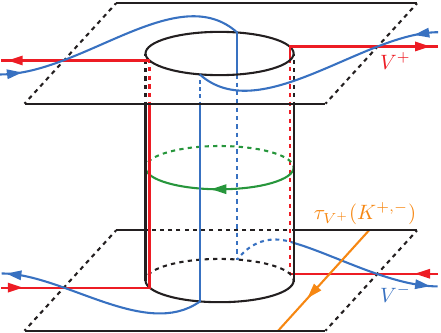}
\centering
\caption{$\tau_{V^{+}}(K^{+,-})$.}
\end{subfigure}
\begin{subfigure}[t]{0.48\textwidth}
\includegraphics[scale=0.7]{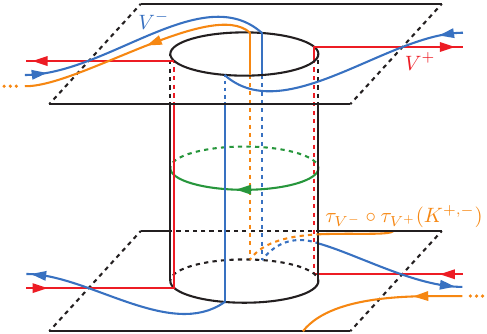}
\centering
\caption{$\tau_{V^{-}} \circ \tau_{V^{+}}(K^{+,-})$.}
\end{subfigure}
\centering
\caption{Monodromy image of $K^{+,-}$.}
\label{fig:K+-mon}
\end{figure}

Lastly, the basic arc $K^{0,-}$ does not intersect any $+$ vanishing cycles and 
then the arc $\tau_{V^{0}}(K^{0,-})$ meets one $V^{-}$ corresponding $-$ vertex 
of given edge (which is not a relevant vanishing cycle). See \cref{fig:K0-mon} 
for the monodromy image $\varphi_{f}(K^{0,-})$.

\begin{figure}[h]
\begin{subfigure}[t]{0.48\textwidth}
\includegraphics[scale=0.7]{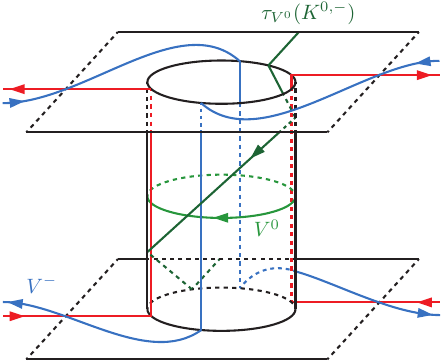}
\centering
\caption{$\tau_{V^{0}}(K^{0,-})$.}
\end{subfigure}
\begin{subfigure}[t]{0.48\textwidth}
\includegraphics[scale=0.7]{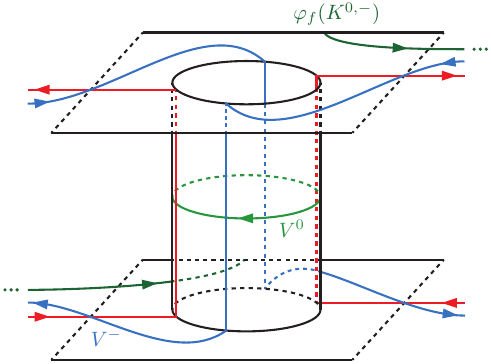}
\centering
\caption{$\varphi_{f}(K^{0,-})$.}
\end{subfigure}
\centering
\caption{Monodromy image of $K^{0,-}$.}
\label{fig:K0-mon}
\end{figure}

From these observations, we characterize basic arcs in the following way.

\begin{prop}\label{prop:key}
The basic arcs $K^{+,0}$, $K^{+,-}$, and $K^{0,-}$ satisfy
$$V_f([K^{+,0}]) = [V^{+}] - [V^{0}], \;\; V_f([K^{+,-}]) = [V^{+}] - [V^{-}], 
\mbox{ and }\; V_f([K^{0,-}]) = [V^{0}] - [V^{-}].$$
\end{prop}
\begin{proof}
This can be checked from the above figures. For example, in the second case, 
 the basic arc $K^{+,-}$ looks like $V^{+}$ with the opposite orientation away 
 from its boundary and similarly, its monodromy image $\varphi_{f}(K^{+,-})$ 
 looks like $V^{-}$ with the opposite orientation. Since they have the same 
 boundary, by definition, the variation image of $[K^{+,-}]$ is $[V^{+}] - 
 [V^{-}]$.
\end{proof}

\section{Construction of arcsets in general}\label{sec:arcsetgeneral}
Given an $A\Gamma$ diagram diagram $A\Gamma(\mathbb{D}_{f})$,  we choose a path 
$\gamma_v$ for each vertex $v$ in $A\Gamma(\mathbb{D}_{f})$ from an outer 
vertex (of depth 0) to the vertex $v$ as follows.

First, we set up our notations. From now on, when we consider a path $\gamma$ 
in the $A\Gamma$ diagram $A\Gamma(\mathbb{D}_{f})$, $\gamma$ is always given by 
the concatenation of all distinct edges $e_{1} e_{2} \dots e_{m}$. Then, we can 
orient these edges in the path $\gamma$ naturally from the starting point to 
the endpoint. For an edge $e_{i}$ contained in $\gamma$, let us denote its 
source and target by $s(e_{i})$ and $t(e_{i})$, respectively. In particular, we 
define the source and target of $\gamma$ by $s(\gamma) := s(e_{1})$ and 
$t(\gamma) := t(e_{m})$. In fact, there is an ambiguity when $\gamma$ consists 
of only one edge. However, we will deal with length 1 paths in 
$A\Gamma(\mathbb{D}_{f})$ such that two vertices of that edge have different 
depths. Then, we define $s(\gamma)$ to be a vertex of smaller depth and 
$t(\gamma)$ to be the other.

\begin{definition}[Good paths] \label{definition:paths}
Given an $A\Gamma(\mathbb{D}_{f})$, we construct paths $\gamma_v$ for all 
vertices $v$ inductively as follows (similar to \Cref{definition:depth}).
\begin{enumerate}
\item For a vertex $v$ of depth 0, we choose $\gamma_v$ to be the constant path 
$e_{v}$ at $v$.
\item Suppose we have chosen paths $\gamma_v$ for all vertices of depth $< k$, 
$(k \geq 1)$.
Then, for a $+$ vertex $v$ of depth $k$, there is a $-$ vertex $w$ of depth 
$k-1$ connected to $v$ by an edge $e$. Similarly, for a $-$ vertex $v$ of depth 
$k$, there is a $+$ vertex $w$ of depth $k-1$ connected to $v$ by an edge $e$. 
Lastly, for a $0$ vertex of depth $k$, there is a vertex $w$ of depth $k-1$ 
connected to $v$ by an edge $e$ such that $w$ is either $+$ or $-$ vertex.
We concatenate the path $\gamma_w$ with the edge $e$ to obtain the path 
$\gamma_v$.
\end{enumerate}
This gives a collection of  paths for every vertex of  
$A\Gamma(\mathbb{D}_{f})$. We call them {\em good paths}.
\end{definition}
A choice of good paths is not unique. We will choose one and fix it from now on.
To write a good path for a vertex of depth $>0$, we will use a notation 
$\gamma_v = e_{1} \dots e_{m}$ where each $e_i$ is an edge of 
$A\Gamma(\mathbb{D}_{f})$ for any $i$.
Thus in this expression of $\gamma_v$, a factor of constant path does not 
appear.

The following is easy to check.
\begin{lemma} \label{lem:paths}
For good paths, the following holds.
\begin{enumerate}
\item For any $+$ and $-$ vertices of depth $\geq 1$, a good path $\gamma_v = 
e_{1} \dots e_{m}$ consists of edges $e_{i}$ such that $\{ |s(e_{i})|, 
|t(e_{i})| \} = \{ + , - \}$ for $1 \leq i \leq m$.
\item For any 0 vertex of depth $\geq 1$, a good path $\gamma_v = e_{1} \dots 
e_{m}$ consists of edges $e_{i}$ such that $\{ |s(e_{i})|, |t(e_{i})| \} = \{ + 
, - \}$ for $1 \leq i \leq m-1$.
\item For any two vertices $v \neq w$, $\gamma_v$ and $\gamma_w$ are either 
disjoint or overlap up to depth $\leq k$ vertices (and edges between them) for 
some $k \leq \mathrm{min} \{ \mathrm{dep}\,v, \mathrm{dep}\, w \}$.
\end{enumerate}
\end{lemma}

Now, we will find an arcset for the vertex $v$ using the chosen good path 
$\gamma_v$. For an edge $e$ in $A\Gamma(\mathbb{D}_{f})$, $K_{e}$ denotes the 
basic arc $K^{|s(e)|,|t(e)|}$ associated to $e$ in \Cref{lem:+0-}. If 
necessary, we also denote its signs: $K^{|s(e)|,|t(e)|}_{e}$.

\begin{definition}
For a constant path $\gamma$ at any depth 0 vertex $v$ in 
$A\Gamma(\mathbb{D}_{f})$, we set $K_{\gamma} := K^{|v|}$ where $K^{|v|}$ is an 
arc introduced in \Cref{subsec:depth0}.
For a non-constant path $\gamma=e_{1} \dots e_{m}$ such that $\mathrm{dep}\, 
t(e_{i}) > 0$ for all $i$, we define the corresponding collection of basic arcs 
as follows. If the depth of $s(e_{1})$ is zero, then $K_{\gamma}$ is defined to 
be
$$K_{\gamma} := K^{|s(e_{1})|} \coprod_{i=1}^{m} K_{e_{i}}.$$
Otherwise, if the depth of $s(e_{1})$ is nonzero, then
$$K_{\gamma} := \coprod_{i=1}^{m} K_{e_{i}}.$$
\end{definition}

\begin{example}
Suppose that a divide is given by  \cref{fig:arcex} locally. Assume that the 
right-most $+$ region is of depth $3$ counting from the left-most $-$ region 
and $\gamma$ is a length $3$ path $e_{1} e_{2} e_{3}$ from that $-$ region to 
the right-most $+$ region. Then, $K_{\gamma}$ consists of $4$ disjoint basic 
arcs, which are two blue arcs and two orange arcs in \cref{fig:arcex}.
Red circle on the right is $V^+$.
\begin{figure}[h]
\includegraphics[scale=1]{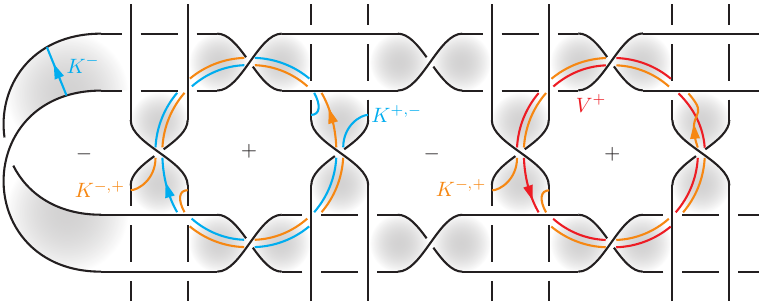}
\centering
\caption{$K_{\gamma}$ for depth $3$ $+$ vertex.}
\label{fig:arcex}
\end{figure}
\end{example}

%

We compute the intersection of the arcset $K_{\gamma_{v}}$ associated to the 
good path $\gamma_{v}$ and other vanishing cycles. 
Note that $K_{\gamma_{v}}$ may have several connected components. 
We will partition these components into smaller groups so that each group is  
one of the following: 
$$K^{-},  K^{-,+}, K^{-,0}, (K^{+} \sqcup K^{+,-}), (K^{-,+} \sqcup K^{+,-}), 
(K^{-,+} \sqcup K^{+,0}).$$

%

More precisely, for  8 kinds of non-constant good paths according to 
$|s(\gamma)|$,$|s(e_{m})|$, and $|t(e_{m})=t(\gamma)|$, 
$\gamma_{v}$ is decomposed as follows:
\begin{itemize}
\item $|s(\gamma)| = +$, $|t(\gamma)| = +$ : $m$ is even. $\gamma_{v}$ and 
$K_{\gamma_{v}}$ are given by
$$\gamma_{v} = e_{1}(e_{2}e_{3}) \dots (e_{m-2}e_{m-1})e_{m},$$
$$K_{\gamma_{v}} = (K^{|s(e_{1})|} \sqcup K_{e_{1}}) \sqcup (K_{e_{2}} \sqcup 
K_{e_{3}}) \dots (K_{e_{m-2}} \sqcup K_{e_{m-1}}) \sqcup K_{e_{m}}.$$
\item $|s(\gamma)| = +$, $|s(e_{m})|=+$, $|t(\gamma)| = 0$ : $m$ is odd. 
$\gamma_{v}$ and $K_{\gamma_{v}}$ are given by
$$\gamma_{v} = e_{1}(e_{2}e_{3}) \dots (e_{m-3}e_{m-2})(e_{m-1}e_{m}),$$
$$K_{\gamma_{v}} = (K^{|s(e_{1})|} \sqcup K_{e_{1}}) \sqcup (K_{e_{2}} \sqcup 
K_{e_{3}}) \dots (K_{e_{m-3}} \sqcup K_{e_{m-2}}) \sqcup (K_{e_{m-1}} \sqcup 
K_{e_{m}}).$$
\item $|s(\gamma)| = +$, $|s(e_{m})|=-$, $|t(\gamma)| = 0$ : $m$ is even. 
$\gamma_{v}$ and $K_{\gamma_{v}}$ are given by
$$\gamma_{v} = e_{1}(e_{2}e_{3}) \dots (e_{m-2}e_{m-1})e_{m},$$
$$K_{\gamma_{v}} = (K^{|s(e_{1})|} \sqcup K_{e_{1}}) \sqcup (K_{e_{2}} \sqcup 
K_{e_{3}}) \dots (K_{e_{m-2}} \sqcup K_{e_{m-1}}) \sqcup K_{e_{m}}.$$
\item $|s(\gamma)| = +$, $|t(\gamma)| = -$ : $m$ is odd. $\gamma_{v}$ and 
$K_{\gamma_{v}}$ are given by
$$\gamma_{v} = e_{1}(e_{2}e_{3}) \dots (e_{m-1}e_{m}),$$
$$K_{\gamma_{v}} = (K^{|s(e_{1})|} \sqcup K_{e_{1}}) \sqcup (K_{e_{2}} \sqcup 
K_{e_{3}}) \dots (K_{e_{m-1}} \sqcup K_{e_{m}}).$$
\item $|s(\gamma)| = -$, $|t(\gamma)| = +$ : $m$ is odd. $\gamma_{v}$ and 
$K_{\gamma_{v}}$ are given by
$$\gamma_{v} = (e_{1}e_{2}) \dots (e_{m-2}e_{m-1})e_{m},$$
$$K_{\gamma_{v}} = K^{|s(e_{1})|} \sqcup (K_{e_{1}} \sqcup K_{e_{2}}) \dots 
(K_{e_{m-2}} \sqcup K_{e_{m-1}}) \sqcup K_{e_{m}}.$$
\item $|s(\gamma)| = -$, $|s(e_{m})|=+$, $|t(\gamma)| = 0$ : $m$ is even. 
$\gamma_{v}$ and $K_{\gamma_{v}}$ are given by
$$\gamma_{v} = (e_{1}e_{2}) \dots (e_{m-1}e_{m}),$$
$$K_{\gamma_{v}} = K^{|s(e_{1})|} \sqcup (K_{e_{1}} \sqcup K_{e_{2}}) \dots 
(K_{e_{m-1}} \sqcup K_{e_{m}}).$$
\item $|s(\gamma)| = -$, $|s(e_{m})|=-$, $|t(\gamma)| = 0$ : $m$ is odd. 
$\gamma_{v}$ and $K_{\gamma_{v}}$ are given by
$$\gamma_{v} = (e_{1}e_{2}) \dots (e_{m-2}e_{m-1})e_{m},$$
$$K_{\gamma_{v}} = K^{|s(e_{1})|} \sqcup (K_{e_{1}} \sqcup K_{e_{2}}) \dots 
(K_{e_{m-2}} \sqcup K_{e_{m-1}}) \sqcup K_{e_{m}}.$$
\item $|s(\gamma)| = -$, $|t(\gamma)| = -$ : $m$ is even. $\gamma_{v}$ and 
$K_{\gamma_{v}}$ are given by
$$\gamma_{v} = (e_{1}e_{2}) \dots (e_{m-1}e_{m}),$$
$$K_{\gamma_{v}} = K^{|s(e_{1})|} \sqcup (K_{e_{1}} \sqcup K_{e_{2}}) \dots 
(K_{e_{m-1}} \sqcup K_{e_{m}}).$$
\end{itemize}

%
We have partitioned them so that each group has nice intersection properties  
(see the case of $(K^{-,+} \sqcup K^{+,-})$  in  \cref{fig:arcex}). 

\begin{lemma} \label{lem:canceled}
Consider a path $\gamma$ in $A\Gamma(\mathbb{D}_{f})$.
\begin{enumerate}
\item Suppose that $\gamma = e_{1}$, $\mathrm{dep}\, s(\gamma)=0$, 
$|s(e_{1})|=+$, and $|t(e_{1})|=-$ for some edge $e_{1}$. $K_{\gamma}$ is 
defined to be $K_{s(e_{1})}^{+} \sqcup K_{e_{1}}^{+,-}$. Then, $K_{\gamma} 
\bullet V_{t(e_{1})} = 1$ and $K_{\gamma} \bullet V = 0$ for any other 
vanishing cycle $V$.
\item Suppose that $\gamma = e_{1}e_{2}$, $\mathrm{dep}\, s(\gamma) \geq 1$, 
$|s(e_{1})|=-$, $|t(e_{1})|=+$, and $|t(e_{2})|=-$. $K_{\gamma}$ is defined to 
be $K_{e_{1}}^{-,+} \sqcup K_{e_{2}}^{+,-}$. Then, $K_{\gamma} \bullet 
V_{s(e_{1})} = -1$, $K_{\gamma} \bullet V_{t(e_{2})} = 1$, and $K_{\gamma} 
\bullet V = 0$ for any other vanishing cycle $V$.
\item Suppose that $\gamma = e_{1}e_{2}$, $\mathrm{dep}\, s(\gamma) \geq 1$, 
$|s(e_{1})|=-$, $|t(e_{1})|=+$, and $|t(e_{2})|=0$. $K_{\gamma}$ is defined to 
be $K_{e_{1}}^{-,+} \sqcup K_{e_{2}}^{+,0}$. Then, $K_{\gamma} \bullet 
V_{s(e_{1})} = -1$, $K_{\gamma} \bullet V = 1$ for $V \in \mathcal{R}_{e_{1}} 
\setminus \mathcal{R}_{e_{2}}$, and $K_{\gamma} \bullet V = 0$ for any other 
vanishing cycle $V$.
\end{enumerate}

\begin{proof}
Let us consider the second case. Note that $K_{e_{1}}^{-,+}$ intersects 
positively any relevant vanishing cycle and $K_{e_{2}}^{+,-}$ intersects 
negatively any relevant vanishing cycle. If a vanishing cycle $V$ is in both 
$\mathcal{R}_{e_{1}}$ and $\mathcal{R}_{e_{2}}$, then $K_{\gamma} \bullet V = 
0$. Hence, we get the following:
$$K_{\gamma} \bullet V = \begin{cases}
1 & (V \in \mathcal{R}_{e_{1}} \setminus \mathcal{R}_{e_{2}}), \\
-1 & (V \in \mathcal{R}_{e_{2}} \setminus \mathcal{R}_{e_{1}}), \\
0 & (V \in \mathcal{R}_{e_{1}} \cap \mathcal{R}_{e_{2}}), \\
0 & (V \notin \mathcal{R}_{e_{1}} \cup \mathcal{R}_{e_{2}}).
\end{cases}$$
One can check that $\mathcal{R}_{e_{1}} \setminus \mathcal{R}_{e_{2}} = 
\{V_{t(e_{2})} \}$ and $\mathcal{R}_{e_{2}} \setminus \mathcal{R}_{e_{1}} = 
\{V_{s(e_{1})} \}$ from \Cref{definition:relevant}. The other cases can be 
proved similarly.
\end{proof}
\end{lemma}

Let us summarize what we have done.
Given an A'Campo divide and its associated  $A\Gamma(\mathbb{D}_{f})$, we have 
chosen good paths $\gamma_v$'s for all vertices $v$'s.
Along $\gamma_v$, we have chosen a family of disjoint properly embedded arcs to 
obtain an arcset $K_{v} := K_{\gamma_{v}}$.

Now, we plan to apply \Cref{prop:excep} to produce a desired topological 
exceptional collection. 
We need to make three assumptions of the proposition to hold. Let us first work 
on the second condition (ii), the adapted condition (see 
\Cref{definition:adapted}).

\begin{prop} \label{thm:adapted}
The family of arcsets $\overrightarrow{K}_{f} = (K^{-}_{1}, \dots, 
K^{-}_{n_{-}}, K^{0}_{1}, \dots, K^{0}_{n_{0}}, K^{+}_{1}, \dots, 
K^{+}_{n_{+}})$ is adapted to the distinguished collection 
$\overrightarrow{V}_{f} = (V^{-}_{1}, \dots, V^{-}_{n_{-}}, V^{0}_{1}, \dots, 
V^{0}_{n_{0}}, V^{+}_{1}, \dots, V^{+}_{n_{+}})$.
\end{prop}
\begin{proof}
Let $v$ be a vertex in $A\Gamma(\mathbb{D}_{f})$. Then, we will show that the 
arcset $K_{v} := K_{\gamma_{v}}$ satisfies all intersection conditions in 
\Cref{definition:adapted}.

First, we consider the case where $|s(\gamma)| = +$, $|t(\gamma)=v| = +$. In 
this case, $K_{\gamma_{v}}$ is decomposed into $(K_{s(e_{1})}^{+} \sqcup 
K_{e_{1}}^{+,-}) \sqcup (K_{e_{2}}^{-,+} \sqcup K_{e_{3}}^{+,-}) \dots 
(K_{e_{m-2}}^{-,+} \sqcup K_{e_{m-1}}^{+,-}) \sqcup K_{e_{m}}^{-,+}$. The first 
piece $(K_{s(e_{1})}^{+} \sqcup K_{e_{1}}^{+,-})$ intersects only 
$V_{t(e_{1})}$ and $(K_{s(e_{1})}^{+} \sqcup K_{e_{1}}^{+,-}) \bullet 
V_{t(e_{1})} = 1$ by \Cref{lem:canceled} (i). The next one $(K_{e_{2}}^{-,+} 
\sqcup K_{e_{3}}^{+,-})$ also intersects $V_{t(e_{1})=s(e_{2})}$, but 
$(K_{e_{2}}^{-,+} \sqcup K_{e_{3}}^{+,-}) \bullet V_{t(e_{1})} = -1$ by 
\Cref{lem:canceled} (ii). Since the depth of $s(e_{i})$ increases 
monotonically, the vanishing cycle $V_{t(e_{1})=s(e_{2})}$ only intersect these 
two pieces. Thus, we get $K_{\gamma} \bullet V_{t(e_{1})}  = 0$.

The second piece $(K_{e_{2}}^{-,+} \sqcup K_{e_{3}}^{+,-})$ intersects 
$V_{t(e_{3})}$ positively but it is canceled algebraically by the intersection 
from the third piece (see \Cref{lem:canceled} (ii)). Hence, we have $K_{\gamma} 
\bullet V_{t(e_{3})}  = 0$. In this way, one can see that the vanishing cycle 
$V_{t(e_{2k-1})}$, $1 \leq k \leq \frac{m-2}{2}$, has nontrivial intersections 
with two pieces but they are canceled.

The vanishing cycles in $\mathcal{R}_{e_{m}}$ and $V_{t(e_{m-1})}$ are the 
remaining nontrivial ones. Note that they are exactly the vanishing cycles 
intersecting the vanishing cycle $V_{v}$. As $|v| = +$, we need to show that 
$K_{\gamma_{v}} \bullet V = V_{v} \bullet V = 1$ for any vanishing cycle $V$ in 
$\mathcal{R}_{e_{m}}$ or $V_{t(e_{m-1})}$, which is induced from (ii) in 
\Cref{thm:ACampo} and (i) in \Cref{definition:adapted}. For $V_{t(e_{m-1})}$, 
it only belongs to $\mathcal{R}_{e_{m-2}}$ and $K_{\gamma_{v}} \bullet 
V_{t(e_{m-1})} = K_{e_{m-2}}^{-,+} \bullet V_{t(e_{m-1})} = 1$. For $V$ in 
$\mathcal{R}_{e_{m}}$, it may belong to $\mathcal{R}_{e_{m-2}}$ and 
$\mathcal{R}_{e_{m-1}}$ also. But these two intersections are canceled as in 
\Cref{lem:canceled}. Therefore, $K_{\gamma_{v}} \bullet V = K_{e_{m}}^{-,+} 
\bullet V = 1$ holds for any $V$ by construction. These arguments prove (i) in 
\Cref{definition:adapted}.

The condition (ii) in \Cref{definition:adapted} follows from that there are 
only $0$ and $-$ vertices in $\mathcal{R}_{e_{m}}$ and $V_{t(e_{m-1})}$ 
corresponds to the $-$ vertex. Hence, $K_{\gamma_{v}}$ does not intersect any 
vanishing cycle $V$ corresponding to some $+$ vertex.

Lastly, since the vanishing cycle $V_{v}$ is only relevant to the edge $e_{m}$, 
the condition (iii) in \Cref{definition:adapted} holds by \Cref{lem:+0-}. 
Therefore, we prove the statement for the first case where $|s(\gamma)| = +$, 
$|t(\gamma)=v| = +$.

Next, we consider the second case where $|s(\gamma)| = +$, $|s(e_{m})|=+$, 
$|t(\gamma) = v| = 0$. Then, $K_{\gamma_{v}}$ is given by $(K_{s(e_{1})}^{+} 
\sqcup K_{e_{1}}^{+,-}) \sqcup (K_{e_{2}}^{-,+} \sqcup K_{e_{3}}^{+,-}) \dots 
(K_{e_{m-1}}^{-,+} \sqcup K_{e_{m}}^{+,0})$. The proof in the first case work 
similarly so that most intersections are canceled algebraically. Therefore, we 
need to consider the vanishing cycles intersecting the last piece 
$(K_{e_{m-1}}^{-,+} \sqcup K_{e_{m}}^{+,0})$ only. By (iii) of 
\Cref{lem:canceled}, such vanishing cycles are in $\mathcal{R}_{e_{m-1}} 
\setminus \mathcal{R}_{e_{m}}$. This set $\mathcal{R}_{e_{m-1}} \setminus 
\mathcal{R}_{e_{m}}$ consists of the $0$ vertex $t(\gamma)$ and two $-$ 
vertices adjacent to $t(\gamma)$. Since $v$ is the $0$ vertex, it implies 
exactly the conditions (i) and (iii) in \Cref{definition:adapted} and the 
condition (ii) follows from the fact $K_{\gamma_{v}} \bullet V = 0$ for any 
vanishing cycle $V \notin \mathcal{R}_{e_{m-1}} \setminus \mathcal{R}_{e_{m}}$.

The last case we prove is $|s(\gamma)| = -$, $|t(\gamma)=v| = -$. In this case, 
$K_{\gamma_{v}} = K^{-}_{s(e_{1})} \sqcup (K_{e_{1}}^{-,+} \sqcup 
K_{e_{2}}^{+,-}) \dots (K_{e_{m-1}}^{-,+} \sqcup K_{e_{m}}^{+,-})$. 
$K^{-}_{s(e_{1})}$ only intersects $V_{s(e_{1})}$ satisfying $K^{-}_{s(e_{1})} 
\bullet V_{s(e_{1})} = 1$ by \Cref{lem:+-} but this intersection is canceled by 
the second piece $(K_{e_{1}}^{-,+} \sqcup K_{e_{2}}^{+,-})$. Thus, similarly, 
the only nontrivial vanishing cycle is $V_{v}^{-}$ intersecting the last piece 
$(K_{e_{m-1}}^{-,+} \sqcup K_{e_{m}}^{+,-})$. By (ii) of \Cref{lem:canceled}, 
$K_{\gamma_{v}} \bullet V_{v}^{-} = 1$, which implies (iii) in 
\Cref{definition:adapted}, and $K_{\gamma_{v}} \bullet V = 0$ for the other 
vanishing cycle $V$, which gives the other conditions.

There are five remaining cases, but they can be proved in a similar way by a 
combination of the above arguments.
\end{proof}

%
%

Next, we work on the assumption (iii) of \Cref{prop:excep}. Namely, we show 
that $K_v$ is a geometric vanishing arcset for every vertex $v$.
For this, it is enough to show that $K_v$ is linear (see 
\Cref{definition:linear}) 
by \Cref{lem:linear}.


We start with the following intersection computations between basic arcs and 
their monodromy images. When we will consider a good path and corresponding 
arcset, these serve as local descriptions about possible intersections between 
its components and monodromy images. Here we only consider intersection points 
away from the boundary $\partial M_{f}$.
\begin{lemma} \label{lem:good}
Let $\gamma$ be a  path of  distinct edges such that $\mathrm{dep} \; s(e) < 
\mathrm{dep} \; t(e)$ for any edge $e$ consisting of $\gamma$ (or simply, be a subpath of any good path).
\begin{enumerate}
\item Let $\gamma = e_{1} e_{2} e_{3} e_{4}$ such that $|s(e_{1})| = +$, 
$|s(e_{2})| = -$, $|s(e_{3})| = +$, $|s(e_{4})| = -$ and $|t(e_{4})| = +$. 
Then, there are two intersection points between basic arcs and their monodromy 
images. $\varphi_{f}(K_{e_{1}}^{+,-})$ only intersects $K_{e_{3}^{+,-}}$ once 
and $\varphi_{f}(K_{e_{1}}^{+,-}) \bullet K_{e_{3}}^{+,-} = -1$. Similarly, 
$\varphi_{f}(K_{e_{4}}^{-,+})$ only intersects $K_{e_{2}}^{-,+}$ once and 
$\varphi_{f}(K_{e_{4}}^{-,+}) \bullet K_{e_{2}}^{-,+} = -1$.
\item Let $\gamma = e_{1} e_{2} e_{3}$ such that $|s(e_{1})| = +$, $|s(e_{2})| 
= -$, $|s(e_{3})| = +$, and $|t(e_{3})| = 0$. Then, there is one intersection 
point between basic arcs and their monodromy images. 
$\varphi_{f}(K_{e_{1}}^{+,-})$ only intersects $K_{e_{3}}^{+,0}$ once and 
$\varphi_{f}(K_{e_{1}}^{+,-}) \bullet K_{e_{3}}^{+,0} = -1$.
\item Let $\gamma = e_{1} e_{2} e_{3}$ such that $|s(e_{1})| = -$, $|s(e_{2})| 
= +$, $|s(e_{3})| = -$, and $|t(e_{3})| = 0$. Then, there is one intersection 
point between basic arcs and their monodromy images. 
$\varphi_{f}(K_{e_{3}}^{-,0})$ only intersects $K_{e_{1}}^{-,+}$ once and 
$\varphi_{f}(K_{e_{3}}^{-,0}) \bullet K_{e_{1}}^{-,+} = -1$.
\end{enumerate}
\begin{proof}
Recall that the monodromy images of the basic arcs are described after 
\Cref{lem:+-}. Then, all the results follow from those descriptions. We give a 
schematic figure for each case; see \cref{fig:mon1}, \cref{fig:mon2}, and 
\cref{fig:mon3}.
\end{proof}
\end{lemma}

\begin{figure}[h]
\includegraphics[scale=0.8]{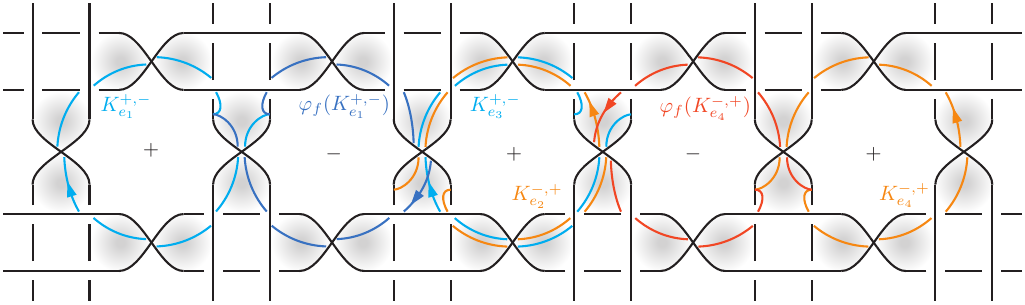}
\centering
\caption{Case (i) in \Cref{lem:good}.}
\label{fig:mon1}
\end{figure}

\begin{figure}[h]
\includegraphics[scale=0.8]{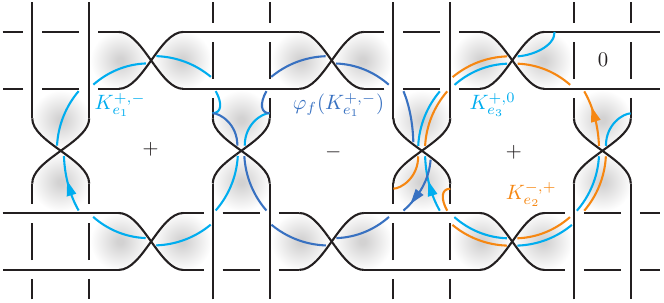}
\centering
\caption{Case (ii) in \Cref{lem:good}.}
\label{fig:mon2}
\end{figure}

\begin{figure}[h]
\includegraphics[scale=0.8]{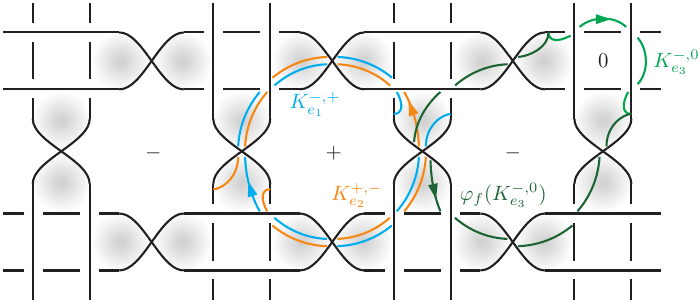}
\centering
\caption{Case (iii) in \Cref{lem:good}.}
\label{fig:mon3}
\end{figure}

\begin{prop} \label{prop:linear}
For any good path $\gamma_{v}$, the vanishing arcset $K_{v}$ is linear. 
Therefore, $K_{v}$ is a geometric vanishing arcset.

\begin{proof}
Let $\gamma_{v} = e_{1} \dots e_{m}$ such that $t(e_{m}) = v$ and $K_{v} = 
K^{|s(e_{1})|} \coprod_{i=1}^{m} K_{e_{i}}$. We already observed in 
\Cref{sec:basic} that any basic arc and its monodromy image intersect only at 
their boundary points.

To show the condition (ii), assume that a good path $\gamma_{v}$ is nonconstant 
(i.e., $K_{v}$ has more than one component). We mainly use \Cref{lem:good} to 
get intersection patterns between $\varphi_{f} (K_{e_{i}})$ and $K_{e_{j}}$ for 
any $1 \leq i \neq j \leq m$. Moreover, we need more intersection data between 
$K^{|s(e_{1})|}$, $K_{e_{1}}$, $K_{e_{2}}$ and their monodromy images. If $m 
\geq 2$, there are two cases: $|s(e_{1})| = -$, $|s(e_{2})| = +$ and 
$|t(e_{2})| = -$, or $|s(e_{1})| = +$, $|s(e_{2})| = -$ and $|t(e_{2})| = +$. 
In the first case, $\varphi_{f}(K_{e_{1}}^{-,+})$ intersects $K^{|s(e_{1})|=-}$ 
once and $\varphi_{f}(K_{e_{1}}^{-,+}) \bullet K^{|s(e_{1})|=-} =-1$. In 
addition, $\varphi_{f}(K^{|s(e_{1})|=-})$ intersects $K_{e_{2}}^{+,-}$ once and 
$\varphi_{f}(K^{|s(e_{1})|=-}) \bullet K_{e_{2}}^{+,-} =-1$. In the second 
case, $\varphi_{f}(K_{e_{2}}^{-,+})$ intersects $K^{|s(e_{1})|=+}$ once and 
$\varphi_{f}(K_{e_{1}}^{-,+}) \bullet K^{|s(e_{1})|=+} =-1$. Also, 
$\varphi_{f}(K^{|s(e_{1})|=+})$ intersects $K_{e_{1}}^{+,-}$ once and 
$\varphi_{f}(K^{|s(e_{1})|=+}) \bullet K_{e_{1}}^{+,-} =-1$. There are no other 
intersections in each case.

There are $8$ kinds of non-constant good paths according to 
$|s(\gamma)|$,$|s(e_{m})|$, and $|t(e_{m})=t(\gamma)|$. For each case, a linear 
order of components is determined by the sign $|s(\gamma)|$ and the number $m$.
\begin{itemize}
\item $|s(\gamma)| = +$, $|t(\gamma)| = +$ : $m$ is even. The linear order is 
given by
$$K_{e_{m}}^{-,+}, \dots, K_{e_{4}}^{-,+}, K_{e_{2}}^{-,+}, K^{|s(e_{1})|=+}, 
K_{e_{1}}^{+,-}, K_{e_{3}}^{+,-}, \dots, K_{e_{m-1}}^{+,-}.$$
\item $|s(\gamma)| = +$, $|s(e_{m})|=+$, $|t(\gamma)| = 0$ : $m$ is odd. The 
linear order is given by
$$K_{e_{m-1}}^{-,+}, \dots, K_{e_{4}}^{-,+}, K_{e_{2}}^{-,+}, K^{|s(e_{1})|=+}, 
K_{e_{1}}^{+,-}, K_{e_{3}}^{+,-}, \dots, K_{e_{m}}^{+,0}.$$
\item $|s(\gamma)| = +$, $|s(e_{m})|=-$, $|t(\gamma)| = 0$ : $m$ is even. The 
linear order is given by
$$K_{e_{m}}^{-,0}, \dots, K_{e_{4}}^{-,+}, K_{e_{2}}^{-,+}, K^{|s(e_{1})|=+}, 
K_{e_{1}}^{+,-}, K_{e_{3}}^{+,-}, \dots, K_{e_{m-1}}^{+,-}.$$
\item $|s(\gamma)| = +$, $|t(\gamma)| = -$ : $m$ is odd. The linear order is 
given by
$$K_{e_{m-1}}^{-,+}, \dots, K_{e_{4}}^{-,+}, K_{e_{2}}^{-,+}, K^{|s(e_{1})|=+}, 
K_{e_{1}}^{+,-}, K_{e_{3}}^{+,-}, \dots, K_{e_{m}}^{+,-}.$$
\item $|s(\gamma)| = -$, $|t(\gamma)| = +$ : $m$ is odd. The linear order is 
given by
$$K_{e_{m}}^{-,+}, \dots, K_{e_{3}}^{-,+}, K_{e_{1}}^{-,+}, K^{|s(e_{1})|=-}, 
K_{e_{2}}^{+,-}, K_{e_{4}}^{+,-}, \dots, K_{e_{m-1}}^{+,-}.$$
\item $|s(\gamma)| = -$, $|s(e_{m})|=+$, $|t(\gamma)| = 0$ : $m$ is even. The 
linear order is given by
$$K_{e_{m-1}}^{-,+}, \dots, K_{e_{3}}^{-,+}, K_{e_{1}}^{-,+}, K^{|s(e_{1})|=-}, 
K_{e_{2}}^{+,-}, K_{e_{4}}^{+,-}, \dots, K_{e_{m}}^{+,0}.$$
\item $|s(\gamma)| = -$, $|s(e_{m})|=-$, $|t(\gamma)| = 0$ : $m$ is odd. The 
linear order is given by
$$K_{e_{m}}^{-,0}, \dots, K_{e_{3}}^{-,+}, K_{e_{1}}^{-,+}, K^{|s(e_{1})|=-}, 
K_{e_{2}}^{+,-}, K_{e_{4}}^{+,-}, \dots, K_{e_{m-1}}^{+,-}.$$
\item $|s(\gamma)| = -$, $|t(\gamma)| = -$ : $m$ is even. The linear order is 
given by
$$K_{e_{m-1}}^{-,+}, \dots, K_{e_{3}}^{-,+}, K_{e_{1}}^{-,+}, K^{|s(e_{1})|=-}, 
K_{e_{2}}^{+,-}, K_{e_{4}}^{+,-}, \dots, K_{e_{m}}^{+,-}.$$
\end{itemize}
Using \Cref{lem:good}, one can show the linearity except $K^{|s(e_{1})|}$, 
$K_{e_{1}}$, and $K_{e_{2}}$. For example, we get that the monodromy image of 
$K_{e_{k}}^{-,+}$ intersects $K_{e_{k-2}}^{-,+}$ once (for any $k \geq 3 \mbox{ 
or }4$ according to the type of $\gamma$). Moreover, by definitions of depth 
and good path, it is enough to consider $K_{e_{k-1}}^{+,-}$ and 
$K_{e_{k+1}}^{+,-}$ for possible intersections. But, also by \Cref{lem:good}, 
they do not intersect the monodromy image of $K_{e_{k}}^{-,+}$. there is no 
other intersection.

In this way, we also have that the monodromy image of $K_{e_{k}}^{+,-}$ only 
intersects $K_{e_{k+2}}^{+,-}$ once, the monodromy image of $K_{e_{m}}^{-,0}$ 
only intersects $K_{e_{m-2}}^{-,+}$ once, and so on. The remaining linearity 
about $K^{|s(e_{1})|}$, $K_{e_{1}}$, and $K_{e_{2}}$ follows from the above 
discussion given in the proof.

Thus, any $K_{v}$ is a linear arcset and by \Cref{lem:linear}, $K_{v}$ is a 
geometric vanishing arcset.
\end{proof}
\end{prop}

\begin{thm} \label{thm:excep}
The adapted family $\overrightarrow{K}_{f}$ in \Cref{thm:adapted} is a 
topological exceptional collection.
\begin{proof}
There is one remaining assumption of \Cref{prop:excep} that all the arcs in the 
arcsets are disjoint.
For this, we need to translate arcs a little bit as follows.
First, for any vertex $w$ on a good path $\gamma_{v}$ for some vertex $v$, its 
good path $\gamma_{w}$ is the subpath of $\gamma_{v}$ by construction and then 
the arc $K_{w}$ is a subset of the arc $K_{v}$. Thus for the common edges,  the 
same basic arcs were chosen. But we can choose mutually disjoint copies of a 
given basic arc by, for example, `translating a little bit' the original basic 
arc (see \cref{fig:copies}) without affecting other intersection conditions. 
Hence, we can assume that any two arcs in $\overrightarrow{K}_{f}$ are disjoint.
Then, by \Cref{prop:linear} and \Cref{prop:excep}, $\overrightarrow{K}_{f}$ is 
a topological exceptional collection.
\end{proof}
\end{thm}

\begin{figure}[h]
\includegraphics[scale=1]{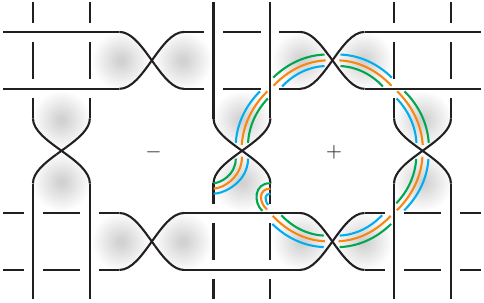}
\centering
\caption{Three disjoint copies of $K^{+,-}$.}
\label{fig:copies}
\end{figure}

\Cref{thm:excep} is the first part of \Cref{thm:maina}. Lastly, for any member 
of $\overrightarrow{K}_{f}$, we will show that its geometric variation image is 
isotopic to the corresponding vanishing cycle described by A'Campo. We perform 
the surgeries at intersection points following the linear order given in 
\Cref{prop:linear}.

Let us focus one specific example of a good path $\gamma = e_{1}e_{2}$ such 
that $|s(\gamma)| = -$, $|t(\gamma)| = -$. $K_{\gamma}$ consists of three basic 
arcs $K^{-}$, $K^{-,+}_{e_{1}}$, and $K^{+,-}_{e_{2}}$. Their linear order is 
given by $K^{-,+}_{e_{1}}$, $K^{-}$, $K^{+,-}_{e_{2}}$. In \cref{fig:isotopy1}, 
the geometric variation images of $K^{-,+}_{e_{1}}$ and $K^{-}$ are given. Then 
one can check that $\mathrm{sg}(\Var_{f}(K^{-,+}_{e_{1}}),\Var_{f}(K^{-}))$ is 
isotopic to the vanishing cycle representing $-[V^{+}]$ where $V^{+}$ 
corresponds to the $+$ vertex $s(e_{2})$ (see the red parts in 
\cref{fig:isotopy1}). Next, two curves 
$\mathrm{sg}(\Var_{f}(K^{-,+}_{e_{1}}),\Var_{f}(K^{-}))$ (after the isotopy) 
and $\Var_{f}(K^{+,-}_{e_{2}})$ are given in \cref{fig:isotopy2}. Similarly, 
its surgery is isotopic to the vanishing cycle $-[V^{-}] = V_{f}(K_{\gamma})$ 
corresponding to $t(\gamma)$. This shows that the geometric variation image 
$\Var_{f}(K_{\gamma})$ is isotopic to the vanishing cycle of A'Campo.

\begin{figure}[h]
\includegraphics[scale=1]{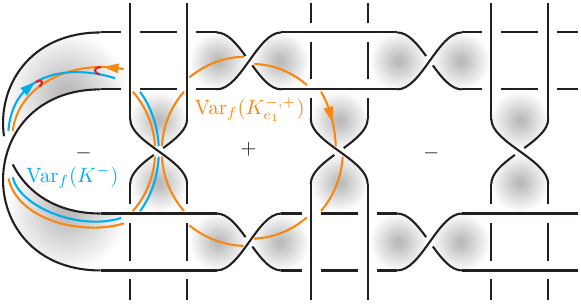}
\centering
\caption{Three curves $\Var_{f}(K^{-,+}_{e_{1}})$, $\Var_{f}(K^{-})$, and 
$\mathrm{sg}(\Var_{f}(K^{-,+}_{e_{1}}),\Var_{f}(K^{-}))$.}
\label{fig:isotopy1}
\end{figure}

\begin{figure}[h]
\includegraphics[scale=1]{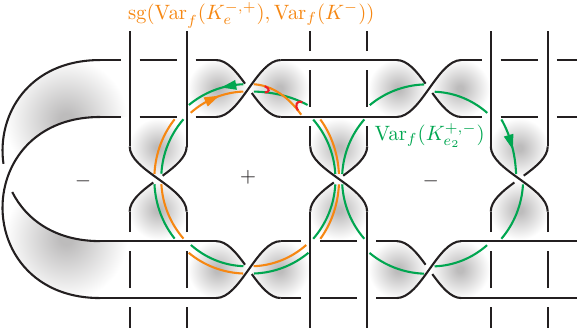}
\centering
\caption{Three curves $\mathrm{sg}(\Var_{f}(K^{-,+}_{e_{1}}),\Var_{f}(K^{-}))$, 
$\Var_{f}(K^{+,-}_{e_{2}})$, and $\Var_{f}(K_{\gamma})$.}
\label{fig:isotopy2}
\end{figure}

Any other case can be shown by repeating the above process because local 
situations regarding any surgeries are always the same (cf. \Cref{prop:key}). 
Roughly, surgeries before a component $K^{|s(e_{1})|}$ give a `long' simple 
closed curve and after that, surgeries make the curve to contract to the 
vanishing cycle corresponding to the vertex $t(\gamma)$. Hence, this proves the 
second part of \Cref{thm:maina}.

\bibliographystyle{alpha}
\bibliography{bibliography}

\end{document}